\documentclass[11pt,reqno]{amsart}
\usepackage{amsmath, amssymb,fullpage}
\usepackage{paralist}
\usepackage{mathrsfs,xcolor}
 \usepackage[colorlinks=true]{hyperref}
\hypersetup{urlcolor=blue, citecolor=blue, linkcolor=blue}

\usepackage{amsfonts}
\usepackage{graphicx}
\usepackage{epstopdf}
\usepackage{algorithmic}
\usepackage{verbatim}
\ifpdf
  \DeclareGraphicsExtensions{.eps,.pdf,.png,.jpg}
\else
  \DeclareGraphicsExtensions{.eps}
\fi

\newtheorem{theorem}{Theorem}

\newtheorem{lemma}[theorem]{Lemma}

\theoremstyle{definition}

\newtheorem{remark}{Remark}

\title[Singular Limits of SWE on the Sphere]{Singular Limits of the\\ Shallow Water
  Equations on the Sphere}

\author{\bf Bin Cheng}
\address{School of Mathematics and Physics, University of Surrey, Guildford, GU2 7XH, United Kingdom}
\email{b.cheng@surrey.ac.uk}

\author{\bf Steve Schochet}
\address{School of Mathematical Sciences, Tel-Aviv University, Tel Aviv 69978, Israel}
  \email{schochet@tauex.tau.ac.il}

\usepackage{amsopn}

\newcommand{\myvec}[1]{\mathbf{#1}}
\newcommand{\tfr}[1]{\tfrac 1{#1}}
\newcommand{\prt}{\partial}
\newcommand{\eqdef}{\mathrel{:=}}
\newcommand{\grad}{\nabla}
\DeclareMathOperator{\curl}{curl}
\let\div\undefined
\DeclareMathOperator{\div}{div}
\newcommand{\eps}{\varepsilon}
\newcommand{\del}{\delta}
\newcommand{\lp}{\left(}\newcommand{\rp}{\right)}
\newcommand{\lb}{\left[}\newcommand{\rb}{\right]}
\newcommand{\ddt}{\tfrac{d\hfil}{dt}}
\newcommand{\ppp}[2]{\frac{\prt #1\hfill}{\prt #2\hfill}}

\newcommand{\tppp}[2]{\tfrac{\prt #1\hfill}{\prt #2\hfill}}
\newcommand{\tpppp}[2]{\tfrac{\prt^2 #1\hfil}{\prt #2^2\hfill}}
\newcommand{\eval}[1]{{\big|_{#1}}}
\newcommand{\ip}[2]{\left\langle {#1},{#2}\right\rangle}
\newcommand{\?}{\kern-1pt}

\DeclareMathOperator{\Av}{A\?v}
\DeclareMathOperator{\Avphi}{A\?v_{\!\phi}}
\newenvironment{extra}{\comment}{\endcomment}
\newif\ifextras\extrasfalse
\ifextras\fi
\ifpdf
\hypersetup{
  pdftitle={Singular Limits of the Shallow Water
  Equations on the Sphere},
  pdfauthor={Bin Cheng and Steve Schochet}
}
\fi

\def\cref{\eqref}
\begin{document}

\maketitle

\begin{abstract}
  Solutions of the slightly compressible shallow water equations on a rapidly rotating sphere are shown to be bounded uniformly when ratio of the
  Froude number to the Rossby number is bounded.  Moreover, in the singular limit in which the ratio of those parameters remains fixed while they both
  tend to zero, solutions with well-prepared initial data tend to corresponding solutions of limit equations. A convergence result is also obtained
  for the three-scale singular limit in which the Froude number tends to zero faster than the Rossby number.
\end{abstract}

\keywords{
Keywords. shallow water equations, rotating sphere, singular limit, uniform bounds, variable-coefficient large operator, three-scale singular limit

MSC codes. 35B25, 35L45, 35Q31, 35Q86, 35R01, 58J45, 76U05
}


\section{Introduction}
In this paper we obtain bounds on solutions of the slightly compressible shallow water equations on a rapidly rotating sphere that are uniform when
the ratio of the Froude number to the Rossby number is bounded.  Similar uniform bounds were obtained for solutions of the incompressible Euler
equations on a rapidly rotating sphere in \cite{MR2998607}, and for solutions of the compressible equatorial $\beta$-plane approximation in
\cite{MR2219357,MR2372481,MR2487851, MR2424189}, which involve only one small parameter. The difficulty in all these problems is that the large terms
have variable coefficients, so the general theory for singular limits of systems having constant-coefficient large terms
\cite{MR615627,MR1303036,MR1459589} is not directly applicable. We follow the approach of \cite{MR2487851} of looking for a modified differential
operator that commutes with the operator appearing in the large terms. Other papers that use that approach include \cite{MR3305659,MR3196189}.

In addition, we show that when the ratio of the Froude number to the Rossby number remains fixed while they both tend to zero then solutions of the
slightly compressible shallow water equations on a rapidly rotating sphere converge to corresponding solutions of limit equations provided that the
initial data is well prepared in the sense that initial time derivative of the solution is uniformly bounded. As for classical singular limits, our
uniform existence result does not require that the initial data be well-prepared. These results are among the few that have been obtained
(\cite{MR2998607,MR3478875,MR3315495} for singular limits on a curved manifold such as the sphere.

The analogous convergence result for the classical
case of constant-coefficient large terms was proven in \cite{MR615627}. In the classical case it has been shown in
\cite{MR1305424,MR1303036,MR1459589} that even when the well-preparedness condition does not hold then after filtering the solutions to remove the
fast oscillations the result tends to solutions of a profile equation, and the corresponding result was shown for the equatorial beta-plane in
\cite{MR2424189,MR3017989}. Similar convergence results were proven in \cite{MR2998607} for the incompressible Euler equations on a rapidly rotating
sphere by taking a time average instead of filtering, using the framework developed in \cite{MR3039689}. Although an abstract version of such a
theorem could be obtained here for the slightly compressible shallow water equations on a rapidly rotating sphere, a truly useful concrete version
would require determining the spectral decomposition of the variable-coefficient large operator~$L$ in \eqref{eq:L} below. 

Furthermore, if instead the ratio of Froude number to the square of the Rossby number remains fixed as those parameters tend to zero, and appropriate
conditions on the initial data hold, then solutions of the slightly compressible shallow water equations on a rapidly rotating sphere again converge to
corresponding solutions of limit equations. As far as we know this is the first result for such a three-scale singular limit
\cite{MR3803773} in which the large terms have variable coefficients.

As in \cite{MR2998607,MR2219357,MR2372481,MR2487851, MR2424189}, the limit solutions here describe zonal flow, whose nature and importance are
described in \cite{MR2998607} and its references.  Due to the specific structure of the equations, the limits obtained here in both the two-scale and
three-scale cases are stationary, unlike the limit in the multi-scale beta-plane \cite{MR2487851} and most other two-scale and three-scale singular
limits. The results of \cite{MR2998607} have been generalized in different ways in \cite{MR3478875,MR3315495}. It would be interesting to try to show
analogous generalizations of the results here.

A precise description of our results is given in \cref{sec:results}. In \cref{sec:operators} the operator obtained from the large terms in the
equations is analyzed, and a differential operator that commutes with that operator is presented.  The uniform boundedness result is proven in
\cref{sec:unif}, and the convergence results are proven in \cref{sec:conv}. An appendix reviews the differential geometry concepts and
notations used throughout the paper.

\section{Equations and Theorem}
\label{sec:results}

The scaled slightly compressible shallow water equations on a rapidly rotating sphere are
\begin{equation}
  \label{eq:euler}
\begin{aligned}
 \prt_t \myvec u+\grad_{\?\myvec u}\myvec u&=-\tfr\del \grad h+\tfrac{z}{\eps}\myvec u^\perp,
\\
\prt_t h+\grad_{\?\myvec u}  h+h\div \myvec u&=-\tfr\del\div \myvec u,
\end{aligned}
\end{equation}
where $\myvec u$ is a vector field \eqref{eq:ucoords} on the sphere, $\grad_{\?\myvec u}$ denotes the covariant derivative \eqref{eq:vgradf},
\eqref{eq:covderiv}, $\myvec u^\perp$ denotes the perpendicular vector field defined by \eqref{eq:uperp}, $z$ is the coordinate in the direction of
the axis of rotation of the sphere \eqref{eq:polar}, and the positive parameters $\del$ and $\eps$ are the Froude number and the Rossby number,
respectively. System~\eqref{eq:euler} differs from the more familiar equatorial beta-plane shallow water equations
(e.g. \cite[(1.1)\textendash(1.2)]{MR2487851}) only in that the covariant derivative $\nabla_{\myvec u}$ replaces the cartesian convective derivative
$\myvec u\cdot\nabla$, the height $z$ on the sphere replaces the height~$y$ in the beta-plane, the divergence, gradient, and perpendicular of a vector
are taken in the appropriate sense for the sphere rather than for the plane, and separate parameters are used for the Froude and Rossby numbers.
Using the coordinates~$\myvec u=\lp\begin{smallmatrix}u_\phi\\u_\theta\end{smallmatrix}\rp$ from \eqref{eq:ucoords} along with the
formulas~\eqref{eq:vgradf} and \eqref{eq:covderiv} for the derivative operator $\grad_{\myvec v}$, the spherical coordinate formula~\eqref{eq:polar},
and formulas~\eqref{eq:uperp} for the perpendicular of a vector, \eqref{eq:div} for the divergence, and \eqref{eq:grad} for the gradient,
\eqref{eq:euler} can be written as
\begin{equation}\label{eq:euler-local}
  \begin{aligned}
    \prt_t &u_\phi+u_\phi\tppp{u_\phi}{\phi}+u_\theta\tppp{u_\phi}{\theta}+2u_\phi u_\theta\cot\?\theta 
=-\tfr\del \tfrac{\ppp{h}{\phi}}{\sin^2\!\theta}+\tfr\eps\tfrac{\cos\?\theta}{\sin\?\theta}\,u_\theta,
\\
    \prt_t &u_\theta+u_\phi\tppp{u_\theta}{\phi}+u_\theta \tppp{u_\theta}{\theta}-u_\phi^2\sin\?\theta\cos\?\theta 
=-\tfr\del \tppp{h}{\theta}-\tfr\eps \cos\?\theta\sin\?\theta\,u_\phi,
\\
    \prt_t &h+u_\phi \tppp{h}{\phi}+u_\theta \tppp{h}{\theta}+h\! \lp \tppp{u_\phi}{\phi}+\tfr{\sin\?\theta}\prt_\theta (\sin\?\theta\,u_\theta)\rp
=-\tfr\del\! \lp \tppp{u_\phi}{\phi}+\tfr{\sin\?\theta}\prt_\theta (\sin\?\theta\,u_\theta)\rp.
  \end{aligned}
\end{equation}

\paragraph{Spaces and norms} The space $L^2(\mathbb S^2)$ is the space of pairs $\lp\begin{smallmatrix}\myvec u\\h\end{smallmatrix}\rp$ of vector
fields and scalar fields for which the norm
\begin{equation*}
  \left\|\lp\begin{smallmatrix}\myvec u\\h\end{smallmatrix}\rp\right\|_{L^2}
\eqdef\sqrt{\ip{\lp\begin{smallmatrix}\myvec u\\h\end{smallmatrix}\rp}{\lp\begin{smallmatrix}\myvec u\\h\end{smallmatrix}\rp}}
\end{equation*}
is finite, where the inner product $\ip{\,}{}$ on the sphere is defined in \eqref{eq:ip}. To define Sobolev norms we need the operator
\cite[(2.1.6)]{MR3478875}
\begin{equation*}
  A\eqdef \lp\begin{smallmatrix}(-\Delta)^{\frac12}&0\\0&(-\Delta)^{\frac12}\end{smallmatrix}\rp,
\end{equation*}
where $\Delta$ in the upper diagonal spot is the Laplacian on vector fields, and $\Delta$ in the lower diagonal spot is the Laplacian on functions.
While \cite{MR3478875}, which does not have the scalar component since it treats the incompressible equations, defines $A$ in terms of the Hodge
Laplacian on one-forms, the definitions are equivalent in view of
\eqref{eq:lap=hodge:f}\textendash\eqref{eq:lap=hodge:u}.  The $H^s$ norm can then be defined by
\begin{equation*}
  \left\|\lp\begin{smallmatrix}\myvec u\\h\end{smallmatrix}\rp\right\|_{H^s}
\eqdef \left\|(I+A^s) \lp\begin{smallmatrix}\myvec u\\ h\end{smallmatrix}\rp\right\|_{L^2},
\end{equation*}
where $I$ denotes the identity operator. The inclusion of $I$ is necessary for the scalar component since constant functions are harmonic, but it is
not really needed for the vector-field component since there are no nonzero harmonic vector fields on the sphere, and \cite{MR3478875} omits it.
Since the sphere is a compact Riemannian manifold, the usual results for Sobolev spaces on $\mathbb R^d$
remain valid (\cite{MR2389638}, \cite{MR3478875}, and references therein); specific estimates will be cited where they are used.

\begin{theorem}[Uniform Existence and Boundedness]\label{thm:unif}
For any finite positive constants $C$ and $K$ there exist finite positive numbers $T(C,K)$ and $B(C,K)$ such that if
\begin{equation}\label{eq:initbnd}
\left\|\lp\begin{smallmatrix}\myvec u_0\\h_0\end{smallmatrix}\rp\right\|_{H^4}\le K
\end{equation}
then for all $\del$ and $\eps$ satisfying 
\begin{equation}
  \label{eq:mucond}
   \tfrac{\del}{\eps}\le C
\end{equation}
 the solution
  $\lp\begin{smallmatrix}\myvec u\\h\end{smallmatrix}\rp$ to \eqref{eq:euler} having initial data
  $\lp\begin{smallmatrix}\myvec u_0\\h_0\end{smallmatrix}\rp$ belongs to $C^0([0,T(C,K)];\allowbreak H^4(\mathbb R^2))$ 
 and satisfies
\begin{equation}\label{eq:unifbnd}
\max_{0\le t\le T(C,K)}\left\|\lp\begin{smallmatrix}\myvec u(t)\\h(t)\end{smallmatrix}\rp\right\|_{H^4}\le B(C,K).
\end{equation}
\end{theorem}

\begin{remark}
\begin{enumerate}
\item The initial data $\lp\begin{smallmatrix}\myvec u_0\\h_0\end{smallmatrix}\rp$ are allowed to depend on $\delta$ and $\eps$ as long as
  \eqref{eq:initbnd} holds with $K$ independent of those parameters.
\item
  The restriction~\eqref{eq:mucond} arises from the fact that positive powers of $\mu\eqdef\frac\del\eps$ appear in the commuting
  operator~\eqref{eq:M} below used to obtain the uniform bound~\eqref{eq:unifbnd}. However, as noted in \cite[Introduction]{MR3803773}, the more
  severe restrictions that usually apply to systems having two small parameters are not needed here because the time-derivative terms are not
  multiplied by coefficients that depend on the dependent variables.
\item
A similar result holds when $H^4$ is replaced by $H^{2s}$ for any integer $s>2$, and corresponding versions of~\cref{thm:converge}
and~\cref{thm:threescale} below then hold. The restriction to Sobolev spaces of even order is due to the fact that the commuting operator~\eqref{eq:M}
obtained below has order two.
\end{enumerate}
\end{remark}

\begin{theorem}[Convergence]\label{thm:converge} Suppose that 
\begin{equation}\label{eq:muconst}
\text{the ratio $\tfrac{\del}{\eps}$ remains constant as $\del$ and $\eps$ tend to zero,}
\end{equation}
and that the initial data has the form
\begin{equation}
  \label{eq:uinit3}
  \lp\begin{smallmatrix}\myvec u\\  h\end{smallmatrix}\rp_{\eval{t=0}}
=\lp\begin{smallmatrix}\myvec u_{0,0}\\h_{0,0}\end{smallmatrix}\rp+\eps \lp\begin{smallmatrix}\myvec u_{0,1}\\h_{0,1}\end{smallmatrix}\rp
+\eps^2\lp\begin{smallmatrix}\myvec u_{0,2}(\eps)\\h_{0,2}(\eps)\end{smallmatrix}\rp,
\end{equation}
where, for some positive constant~$K$,
\begin{equation}
  \label{eq:unit3bnd}
  \left\|\lp\begin{smallmatrix}\myvec u_{0,j}\\h_{0,j}\end{smallmatrix}\rp\right\|_{H^4}\le K, \qquad\text{for $j=0,1,2$}.
\end{equation}
Assume further that 
  \begin{equation}
    \label{eq:wellprep0}
    - \grad h_{0,0}+\mu z \myvec u_{0,0}^\perp=\myvec 0,\qquad \div \myvec u_{0,0}=0,
  \end{equation}
where 
  \begin{equation}
    \label{eq:mu}
    \mu\eqdef\tfrac\del\eps.
  \end{equation}
  Then as $\del,\eps\to0$ the solution to \eqref{eq:euler} with initial data~\eqref{eq:uinit3}
   converges, both in $C^0([0,T(\mu,3K)];H^s(\mathbb S^2))$ for $s<4$ and weak-$*$ in
  $L^\infty([0,T(\mu,3K)];H^4(\mathbb S^2))$ to the solution $\lp\begin{smallmatrix}\overline{\myvec u}\\[1pt] \overline h\end{smallmatrix}\rp$ of
\begin{equation*}
\prt_t\! \lp\begin{smallmatrix}\overline{\myvec u}\\[1pt] \overline h\end{smallmatrix}\rp=0, \qquad
\lp\begin{smallmatrix}\overline{\myvec u}\\[1pt] \overline h\end{smallmatrix}\rp_{\eval t=0}=\lp\begin{smallmatrix}\myvec u_{0,0}\\h_{0,0}\end{smallmatrix}\rp.
\end{equation*}
More specifically, the limit is identically equal to the initial data, which have the form
\begin{equation*}
\overline{\myvec u}\equiv\myvec u_{0,0}=u_{\phi,0}(\theta)\myvec e_\phi 
\quad  \text{and}\quad\overline h\equiv h_{0,0}=\widetilde h_{0}-\mu\!\int_0^\theta \cos\?\beta\,\sin\?\beta\,u_{\phi,0}(\beta)\,d\beta,
\end{equation*}
where $\myvec e_\phi$ is given in \eqref{eq:ucoords} and the function~$u_{\phi,0}(\theta)$ and constant~$\widetilde h_{0}$ are arbitrary.
\end{theorem}

\begin{remark}
  The results in \cite{MR2487851} for the beta-plane model are presented only for the case when the Froude number equals the Rossby number.  However,
  it is not hard to modify the arguments there to show that when the Rossby number~$\eps$ is different from the Froude number~$\del$ then the uniform
  existence result of \cite{MR2487851} remains valid provided that \eqref{eq:mucond} holds and the convergence result there remains valid provided
  that \eqref{eq:muconst} holds.
\end{remark}

\begin{theorem}[Three-scale convergence]\label{thm:threescale}
  Suppose that 
\begin{equation}\label{eq:nuconst}
\text{the ratio $\tfrac{\del}{\eps^2}$ remains constant as $\del$ and $\eps$ tend to zero,}
\end{equation}
and that the initial data has the form~\eqref{eq:uinit3}, where \eqref{eq:unit3bnd} holds.
Assume further that
\begin{equation}
  \label{eq:unit3conds}
\begin{aligned}
  h_{0,0}&=\widetilde h_{0,0}, & &\myvec u_{0,0}=u_{\phi,0,0}(\theta)\myvec e_\phi, 
\\
h_{0,1}&= \widetilde h_{0,1}-\nu\!\int_0^\theta \cos\?\beta\,\sin\?\beta\,u_{\phi,0,0}(\beta)\,d\beta,
& &\div \myvec u_{0,1}=0,
\end{aligned}
\end{equation}
where
\begin{equation}
  \label{eq:nu}
  \nu\eqdef \tfrac{\del}{\eps^2},
\end{equation}
$\myvec e_\phi$ is given in \eqref{eq:ucoords}, and the function~$u_{\phi,0,0}(\theta)$ and constants~$\widetilde h_{0,0}$ and
$\widetilde h_{0,1}$ are arbitrary.

Then as $\del,\eps\to0$ the solution to \eqref{eq:euler} with initial data~\eqref{eq:uinit3}
 converges, both in $C^0([0,T(\mu,3K)];H^s(\mathbb S^2))$ for $s<4$ and weak-$*$ in
$L^\infty([0,T(\mu,3K)];H^4(\mathbb S^2))$ to the limit initial data~$\lp\begin{smallmatrix}\myvec u_{0,0}\\ h_{0,0}\end{smallmatrix}\rp$.
\end{theorem}

\begin{remark}
\begin{enumerate}
\item  Theorems similar to \cref{thm:threescale} can be also obtained under assumptions that $\frac{\del}{\eps^p}$ remains fixed as $\del,\eps\to0$ for any
  integer $p$, but the case $p=2$ in \cref{thm:threescale} is the simplest. The analogous theorem for the case when $p>1$ is not an integer
  would require that the initial data be trivial.
\item
As is often the case for three-scale singular limits, the proof of \cref{thm:threescale} is trickier than the proof of \cref{thm:converge}.
\end{enumerate}
\end{remark}

\section{The Large Operator and a Commuting Differential Operator}
\label{sec:operators}

\paragraph{Large operator and its null space}
The terms in \eqref{eq:euler}\textendash\eqref{eq:euler-local} that are multiplied by $\tfr\del$ or $\tfr\eps$ are linear in the dependent
variables~$U\eqdef\lp\begin{smallmatrix}\myvec u\\ h\end{smallmatrix}\rp$. It will be convenient to let $\mu$ be defined by \eqref{eq:mu} whether that
ratio is fixed or variable, and substitute $\tfr\eps=\tfrac\mu\del$ into \eqref{eq:euler}\textendash\eqref{eq:euler-local}, so that the right sides of
those equations become $\tfr{\del}LU$, where $L$ has the form
\begin{equation}
  \label{eq:L}
  L=L_1+L_2,
\end{equation}
with
\begin{equation}
  \label{eq:L1}
  L_1\eqdef -\lp\begin{smallmatrix} 0 & \grad\\ \div & 0\end{smallmatrix}\rp
=-\lp\begin{smallmatrix}0&0&\tfr{\sin^2\!\theta}\prt_\phi\\0&0&\prt_\theta\\\prt_\phi&\frac1{\sin\?\theta}\prt_\theta\sin\?\theta&0\end{smallmatrix}\rp
\end{equation}
and
\begin{equation}
  \label{eq:L2}
  L_2\eqdef \mu z\lp \begin{smallmatrix}\perp &0\\0&0\end{smallmatrix}\rp
= \lp\begin{smallmatrix}\phantom{-}0&\mu\tfrac{\cos\?\theta}{\sin\?\theta}&0\\-\mu\sin\?\theta\cos\?\theta &0&0\\\phantom{-}0&0&0 \end{smallmatrix}\rp,
\end{equation}
in which formula~\eqref{eq:perp} has been used. Here and later $\frac1{\sin\?\theta}\prt_\theta\sin\?\theta$ denotes the operator whose action on a
function~$f$ is $\frac1{\sin\?\theta}\prt_\theta(\sin\?\theta\, f)$, and the first form of an operator gives it definition on pairs
$\lp\begin{smallmatrix}\myvec u\\h\end{smallmatrix}\rp$, where $\myvec u$ is a vector field on the sphere and $h$ is a scalar field, while the second
form of the operator shows its action on triplets $\lp\begin{smallmatrix}u_\phi\\u_\theta\\h\end{smallmatrix}\rp$, where $u_\phi$ and $u_\theta$ are
the components of the vector field~$\myvec u$ in accordance with~\eqref{eq:ucoords}.  
\begin{lemma}\label{lem:Lanti}
  $L$ is an anti-symmetric operator, i.e.
\begin{equation}
  \label{eq:Lanti}
  \ip{\lp\begin{smallmatrix}\myvec u\\h\end{smallmatrix}\rp}{L\lp\begin{smallmatrix}\myvec u\\h\end{smallmatrix}\rp}=0.
\end{equation}
\end{lemma}
\begin{proof}
  By the definition \eqref{eq:ip} of the inner product,
  \begin{equation*}
    \ip{\lp\begin{smallmatrix}\myvec u\\h\end{smallmatrix}\rp}{L\lp\begin{smallmatrix}\myvec u\\h\end{smallmatrix}\rp}
=-\ip{\lp\begin{smallmatrix}\myvec u\\h\end{smallmatrix}\rp}{\lp\begin{smallmatrix}\grad h\\\div\myvec u\end{smallmatrix}\rp}
+\mu \int_{\mathbb S^2} z g(\myvec u,\myvec u^\perp),
  \end{equation*}
which vanishes on account of \eqref{eq:divstargrad} and \eqref{eq:uperpperpu}.
\end{proof}

As in the classical theory, the null space of $L$ will play two important roles in \cref{thm:converge}. First,
for fixed initial data~$U_0\eqdef\lp\begin{smallmatrix}\myvec u_0\\h_0\end{smallmatrix}\rp$, the well-preparedness condition that the time derivative
at time zero of a solution should be bounded uniformly for $\del\in(0,1]$ holds iff $LU_0=0$. In addition, the limit
$\overline U\eqdef \lim_{\del\to0}U$ of the solutions of \eqref{eq:euler} with well-prepared initial data will be shown to satisfy $L\overline U=0$.
The following result gives a complete description of the null space of $L$.
\begin{lemma}
  \label{lem:N(L)}
A function $U=\lp\begin{smallmatrix}\myvec u\\h\end{smallmatrix}\rp$ satisfies $LU=0$ iff
\begin{equation}
  \label{eq:wellprep}
  \myvec u=u_\phi(\theta)\myvec e_\phi\qquad\text{and}\qquad h=h(\theta)\eqdef c-\mu\int_0^\theta \cos\?\beta\,\sin\?\beta\, u_\phi(\beta)\,d\beta,
\end{equation}
where $\myvec e_\phi$ is the basis vector in the latitudinal direction from~\eqref{eq:ucoords}\textendash\eqref{eq:ephi-etheta}, $u_\phi(\theta)$ is
an arbitrary function of $\theta$, and $c$ is an arbitrary constant.
\end{lemma}

\begin{proof}
By the first forms of the operators in~\eqref{eq:L1}\textendash\eqref{eq:L2}, the condition $LU=0$ can be written as
\begin{equation}
  \label{eq:wellprep1}
  -\grad h+\mu z\myvec u^\perp=\myvec 0\qquad\text{and}\qquad \div \myvec u=0.
\end{equation}
Applying first $\perp$ and then $\div$ to the first equation in \eqref{eq:wellprep1}, adding $z$ times the second equation there to the result, and
using~\eqref{eq:polar} and~\eqref{eq:divgradperp0} implies 
\begin{equation}
  \label{eq:utheta0}
    u_\theta=0.
\end{equation}
On the other hand, dividing the first equation in \eqref{eq:wellprep1} by $z$ and afterwards applying $\perp$ and then $\div$, adding $\mu$ times the
second equation there to the result, and using \eqref{eq:gradperp}, \eqref{eq:grad}, \eqref{eq:uperp}, \eqref{eq:polar}, \eqref{eq:div}, and
\eqref{eq:divgradperp0} implies 
\begin{equation}\label{eq:dhdphi0}
  \tppp{h}{\phi}=0.
\end{equation}
Using formulas \eqref{eq:grad}, \eqref{eq:uperp}, and \eqref{eq:div} to express \eqref{eq:wellprep1} in terms of $h$ and the components $u_\phi$ and
$u_\theta$ of $\myvec u$ from \eqref{eq:ucoords}, and substituting~\eqref{eq:utheta0}\textendash\eqref{eq:dhdphi0} into the results yields
\begin{equation}
  \label{eq:wellprep2}
  -\tppp{h}{\theta}-\mu \cos\?\theta\,\sin\?\theta \, u_\phi=0\qquad\text{and}\qquad \tppp{u_\phi}{\phi}=0.
\end{equation}
Conditions~\eqref{eq:wellprep2} and \eqref{eq:utheta0}\textendash\eqref{eq:dhdphi0} show that
$U=\lp\begin{smallmatrix}\myvec u\\h\end{smallmatrix}\rp$ satisfies $LU=0$ iff \eqref{eq:wellprep} holds.
\end{proof}

\paragraph{Pseudo-projection onto null space of $L$} 
Since the operator~$L$ is antisymmetric, the the $L^2$ orthogonal projection operator $\mathbb P$ onto its null space also maps its range to zero,
i.e., $\mathbb P L=0$. That orthogonal projection operator plays an important role in the calculation of the limit equations for classical singular
limits (e.g. \cite[pp. 9\textendash10]{MR957005}). Here we shall use instead the operator
\begin{equation}
  \label{eq:P}
  P(\mu)\eqdef \lp\begin{smallmatrix} \frac1{\sin^2\!\theta}\,g(\myvec e_\phi,\Avphi \cdot)\myvec e_\phi &0\\
-\mu I_\theta [\frac{\cos\?\theta}{\sin\?\theta} g(\myvec e_\phi,\Avphi\cdot)]&\Av \end{smallmatrix}\rp,
\end{equation}
where $g$ is the metric~\eqref{eq:g} and 
\begin{equation*}
\begin{aligned}
  \Av f&\eqdef \frac{\int_{\mathbb S^2} f}{\int_{\mathbb S^2} 1}=\tfr{4\pi}\!\int_0^\pi\!\int_0^{2\pi} f\sin\?\theta\,d\phi\,d\theta,
\quad \Avphi f\eqdef \tfr{2\pi} \int_{0}^{2\pi}  f\,d\phi,
\\
I_\theta f&\eqdef\int_0^\theta f(\phi,\theta_1)\,d\theta_1.
\end{aligned}
\end{equation*}
The operator~$P(\mu)$ is not an orthogonal projection since $P(\mu)^*$ does not equal~$P(\mu)$, and it is not even a projection at all since $P(\mu)^2$ does not
equal~$P(\mu)$. While both of those defects could be remedied by adding more terms to the operator, it will be sufficient to use the simpler operator~$P(\mu)$
as long as we adjust for those defects where necessary.
\begin{lemma}\label{lem:P}
The range of the operator $P(\mu)$ from \eqref{eq:P} is the null space of $L$, and its adjoint
\begin{equation}
  \label{eq:Pstar}
  P(\mu)^*= \lp\begin{smallmatrix} \frac1{\sin^2\!\theta}\,g(\myvec e_\phi,\Avphi \cdot)\myvec e_\phi\; & 
-\mu \myvec e_\phi \frac{\cos\?\theta}{\sin^2\!\theta} \widehat I_\theta[\sin\?\theta\Avphi \cdot] \\
0&\Av \end{smallmatrix}\rp,
\end{equation}
where
\begin{equation*}
  \widehat I_\theta f\eqdef \int_{\theta}^\pi f(\phi,\theta_1)\,d\theta_1,
\end{equation*}
satisfies 
\begin{equation}
  \label{eq:PL0}
  P(\mu)^*L=0.
\end{equation}
In particular,
\begin{equation}
  \label{eq:P0L10}
  P(0)^*L_1=0 \qquad \text{and}\quad P(0)^*L_2\lp\begin{smallmatrix}\myvec
    u\\h\end{smallmatrix}\rp=\lp\begin{smallmatrix}\mu\frac{\cos\?\theta}{\sin\?\theta}(\Avphi\? u_\theta)\myvec e_\phi\\0\end{smallmatrix}\rp.
\end{equation}
\end{lemma}
\begin{proof}
By the definition~\eqref{eq:P} of $P(\mu)$ and formulas~\eqref{eq:ucoords} and~\eqref{eq:g},
\begin{equation}
  \label{eq:Pu}
\begin{aligned}
 \lp\begin{smallmatrix} \frac1{\sin^2\!\theta}\,g(\myvec e_\phi,\Avphi \cdot)\myvec e_\phi &0\\
-\mu I_\theta [\frac{\cos\?\theta}{\sin\?\theta} g(\myvec e_\phi,\Avphi\cdot)]&\Av \end{smallmatrix}\rp
\lp\begin{smallmatrix}\myvec u\\h\end{smallmatrix}\rp
&=\lp\begin{smallmatrix} \frac1{\sin^2\!\theta}\,g(\myvec e_\phi,\Avphi \cdot)\myvec e_\phi &0\\
-\mu I_\theta [\frac{\cos\?\theta}{\sin\?\theta} g(\myvec e_\phi,\Avphi\cdot)]&\Av \end{smallmatrix}\rp
\lp\begin{smallmatrix}u_{\phi}\myvec e_\phi+u_\theta \myvec e_\theta\\h\end{smallmatrix}\rp
\\&=\lp\begin{smallmatrix}(\Avphi u_\phi)\myvec e_\phi\\
(\Av h)-\mu \int_0^\theta \cos\?\theta_1\sin\?\theta_1 (\Avphi u_\phi(\phi,\theta_1))\,d\theta_1\end{smallmatrix}\rp.
\end{aligned}
\end{equation}
Comparing the final expression for $P(\mu)\lp\begin{smallmatrix}\myvec u\\h\end{smallmatrix}\rp$ in \eqref{eq:Pu} with~\eqref{eq:wellprep} shows that the
range of~$P(\mu)$ lies in the null space of~$L$, and since $u_\phi$ and $h$ are arbitrary it contains the entire null space. 

The proofs below will not actually use the fact that the operator denoted $P(\mu)^*$ in \eqref{eq:Pstar} is the adjoint of $P(\mu)$, so we will just note that
the fact that it is indeed the adjoint can be verified by using \eqref{eq:g} and \eqref{eq:ip}\textendash\eqref{eq:adjcond}, and changing the order of
integration in the term in the inner product in which~$\widehat I_\theta$ appears.

By formula~\eqref{eq:Pstar} for $P(\mu)^*$, the definition~\eqref{eq:L}\textendash\eqref{eq:L2} of~$L$, formulas~\eqref{eq:uperp}, \eqref{eq:grad},
and~\eqref{eq:div} for $\myvec u^\perp$, $\grad h$, and $\div \myvec u$, and the fact that the averages~$\Avphi(\prt_\phi\,\cdot)$,
$\Av(\prt_\phi\,\cdot)$, and $\Av(\frac1{\sin\?\theta}\prt_\theta(\sin\?\theta\, \cdot))$ all vanish,
  \begin{equation}
    \label{eq:PL}
\begin{aligned}
    P(\mu)^*L\lp\begin{smallmatrix}\myvec u\\h\end{smallmatrix}\rp
&=P(\mu)^*
\lp\begin{smallmatrix} \mu z\perp  & -\grad\\ -\div & 0\end{smallmatrix}\rp
\lp\begin{smallmatrix}\myvec u\\h\end{smallmatrix}\rp
=
P(\mu)^*
\lp\begin{smallmatrix}\mu z \myvec u^\perp-\grad h\\ -\div \myvec u\end{smallmatrix}\rp
\\&\hskip-30pt=\lp\begin{smallmatrix} \frac1{\sin^2\!\theta}\,g(\myvec e_\phi,\Avphi \cdot)\myvec e_\phi\; & 
-\mu \myvec e_\phi\frac{\cos\?\theta}{\sin^2\!\theta} \widehat I_\theta[\sin\?\theta\Avphi \cdot] \\
0&\Av \end{smallmatrix}\rp
\lp\begin{smallmatrix} (\mu\frac{\cos\?\theta}{\sin\?\theta}u_\theta-\sin^2\!\theta\ppp{h}{\phi})\myvec e_\phi
-(\mu\sin\?\theta u_\phi+\ppp{h}{\theta})\myvec e_\theta\\ 
-\ppp{u_\phi}{\phi}-\frac1{\sin\?\theta}\prt_\theta(\sin\?\theta\,u_\theta)\end{smallmatrix}\rp
\\&\hskip-30pt=
\lp\begin{smallmatrix}\mu \myvec e_\phi \frac{\cos\?\theta}{\sin\?\theta}\Avphi u_\theta
-\mu \myvec e_\phi \frac{\cos\?\theta}{\sin\?\theta}\Avphi u_\theta
\\0\end{smallmatrix}  \rp=\lp\begin{smallmatrix}\myvec 0\\0\end{smallmatrix}\rp,
\end{aligned}
\end{equation}
which shows that \eqref{eq:PL0} holds. Setting $\mu$ equal to zero in \eqref{eq:PL0} yields the first equation in~\eqref{eq:P0L10}, while the second
is obtained by noting what part of~\eqref{eq:PL} comes from $P(0)^*$ and $L_2$.
\end{proof}

\paragraph{Modified Laplacian operator}
As noted in the introduction, the key to our results is the fact that it is possible to modify the Laplacian by adding zeroth-order terms so as to
obtain an operator $M$ that commutes with~$L$.  Define
\begin{equation}
  \label{eq:M}
  M\eqdef M_1+M_2
\end{equation}
with
\begin{equation}
  \label{eq:M1}
\begin{aligned}
  M_1&\eqdef\lp\begin{smallmatrix}-\Delta& 0\\0&-\Delta\end{smallmatrix}\rp
\\&=\lp\begin{smallmatrix}-\prt_\theta^2 -\csc^2\!\theta\,\prt_\phi^2-3\cot\?\theta\,\prt_\theta+2& -2\cot\?\theta \csc^2\!\theta\,\prt_\phi&0\\ 
2\cot\?\theta\,\prt_\phi& -\prt_\theta^2 -\csc^2\!\theta\,\prt_\phi^2-\cot\?\theta\,\prt_\theta+\csc^2\!\theta\,&0\\
0&0&-\prt_\theta^2 -\csc^2\!\theta\,\prt_\phi^2-\cot\?\theta\,\prt_\theta  \end{smallmatrix}\rp
\end{aligned}
\end{equation}
and
\begin{equation}
  \label{eq:M2}
  M_2\eqdef\lp\begin{smallmatrix}\mu^2(\frac12-\sin^2\!\theta)&2\mu\myvec e_\phi\\2 \mu g(\myvec e_\phi,\cdot)&\mu^2(\frac12-\sin^2\!\theta) \end{smallmatrix}\rp
  =\lp\begin{smallmatrix}
    \mu^2(\frac12-\sin^2\!\theta)&0&2\mu\\0&\mu^2(\frac12-\sin^2\!\theta)&0\\2\mu\sin^2\!\theta&0&\mu^2(\frac12-\sin^2\!\theta)\end{smallmatrix}\rp,
\end{equation}
where \eqref{eq:hodge0}, \eqref{eq:lapu}, and \eqref{eq:sphereip} have been used.
The key observation is that $M$ commutes with the large operator~$L$, i.e.
\begin{equation}
  \label{eq:LMML}
[M,L]\eqdef M L-LM=0.
\end{equation}
We found the operator~$M_2$ needed to obtain~\eqref{eq:LMML} via a calculation in Mathematica\textsuperscript{\textregistered}
\cite{Mathematica}, but of course the result can be verified directly:
\begin{lemma}
  The operators~$L$ from \eqref{eq:L}\textendash\eqref{eq:L2} and $M$ from \eqref{eq:M}\textendash\eqref{eq:M2} satisfy \eqref{eq:LMML}.
\end{lemma}

\begin{proof}
  Using first \eqref{eq:L1}, \eqref{eq:M1}, \eqref{eq:lap=hodge:u}, \eqref{eq:lap=hodge:f}, \eqref{eq:hodge},
  \eqref{eq:gradperp}, and \eqref{eq:divcurl}, and then \eqref{eq:flatsharp}, \eqref{eq:starstar}, and \eqref{eq:dd0} yields
\begin{equation}
  \label{eq:M1L1}
\begin{aligned}
  &[M_1,L_1]
\\&\quad=
\lp\begin{smallmatrix}\sharp(\star d \star d+d\star d\star)\flat&0\\0&\star d \star d+d\star d\star\end{smallmatrix}\rp
\lp\begin{smallmatrix}0&\sharp\star d\\\star d\star\flat&0 \end{smallmatrix}\rp
-\lp\begin{smallmatrix}0&\sharp\star d\\\star d\star\flat&0 \end{smallmatrix}\rp
\lp\begin{smallmatrix}\sharp(\star d \star d+d\star d\star)\flat&0\\0&\star d \star d+d\star d\star\end{smallmatrix}\rp
\\&\quad= \lp\begin{smallmatrix} 0&\sharp\star d\star d\star d\\\star d \star d \star d \star\flat&0\end{smallmatrix}\rp
-\lp\begin{smallmatrix} 0&\sharp\star d\star d\star d\\\star d \star d \star d \star\flat&0\end{smallmatrix}\rp
=0.
\end{aligned}
\end{equation}
On the other hand, by \eqref{eq:L2}, \eqref{eq:M1}, \eqref{eq:lap=hodge:u}, \eqref{eq:lap=hodge:f}, \eqref{eq:hodge}, and \eqref{eq:perp},
\begin{equation*}
  \begin{aligned}
    &[M_1,L_2]
\\&\quad= \lp\begin{smallmatrix}\sharp(\star d \star d+d\star d\star)\flat&0\\0&\star d \star d+d\star d\star \end{smallmatrix}\rp 
\lp\begin{smallmatrix} \mu z\sharp \star \flat&0\\0&0\end{smallmatrix}\rp
-\lp\begin{smallmatrix} \mu z\sharp \star \flat&0\\0&0\end{smallmatrix}\rp 
\lp\begin{smallmatrix} \sharp(\star d \star d+d\star d\star)\flat&0\\0&\star d \star d+d\star d\star\end{smallmatrix}\rp .
  \end{aligned}
\end{equation*}
Since the operators $\flat$, $\sharp$, and $\star$ involve only multiplication, not differentiation, the factor $z$ can be moved past them at will,
i.e. $\star z=z\star$, etc.. In contrast, the product rule for derivatives together with \eqref{eq:polar} shows that $d\,z=zd-\sin\?\theta\, d\theta\wedge$,
where the $\wedge$ is omitted if $d\,z$ is applied to a zero-form rather than a one-form, and $d\theta\wedge$ denotes the operator that when applied
to a one-form~$\alpha$ yields $d\theta\wedge \alpha$.  Therefore 
\begin{equation}
  \label{eq:M1L2b}
  [M_1,L_2]=
    \lp \begin{smallmatrix}
    Q&0\\0&0
  \end{smallmatrix}\rp,
\end{equation}
where
\begin{equation}
  \label{eq:Q}
\begin{aligned}
  Q&=\mu[\sharp(\star\, d \star d+d\star d\,\star)z \star \flat -z\sharp \star(\star\, d\, \star d+d\star d\,\star)\flat ]
\\&=\mu\sharp(\star\, d\star d\,  z\star - d\star d\, z +z\, d\star d-z \star d \star d \,\star)\flat
\\&= \mu\sharp((\star\, d\star z d\star-\star d \star \sin\?\theta\, d\theta\!\wedge \star) 
- (d\star z d-d\star \sin\?\theta\, d\theta \wedge) +z\, d\star d-z \star d \star d \,\star )\flat
\\&=\mu\sharp (-\star \sin\?\theta \,d\theta\star d \star-\star d \star \sin\?\theta\,d\theta\!\wedge\star
+\sin\?\theta\,d\theta \star d +d\star \sin\?\theta\,d\theta\wedge )\flat
\end{aligned}
\end{equation}
and \eqref{eq:starstar} and \eqref{eq:flatsharp} have also been used.
By \eqref{eq:wedgeanti} and \eqref{eq:star20},
\begin{equation}
  \label{eq:starsinwedge}
  \star\sin\?\theta\,d\theta\!\wedge (f d\phi+kd\theta) =-f,
\end{equation}
while by \eqref{eq:flat} and \eqref{eq:star1}
\begin{equation}
  \label{eq:starflat}
  \flat (u_\phi\myvec e_\phi  +u_\theta \myvec e_\theta)=u_\phi\sin^2\!\theta\,d\phi+u_\theta\,d\theta,
\qquad
\star \flat (u_\phi\myvec e_\phi  +u_\theta \myvec e_\theta)=u_\phi\sin\?\theta\,d\theta-\sin\?\theta\,u_\theta\,d\phi.
\end{equation}
Formulas~\eqref{eq:starsinwedge}, \eqref{eq:starflat}, and \eqref{eq:sphereip} show that
\begin{equation}\label{eq:starsinwedge2}
\begin{aligned}
  \star\sin\?\theta\,d\theta\!\wedge \flat (u_\phi\myvec e_\phi  +u_\theta \myvec e_\theta) 
&=  -u_\phi\sin^2\!\theta 
\\
\star\sin\?\theta\,d\theta\!\wedge \star\, \flat (u_\phi\myvec e_\phi  +u_\theta \myvec e_\theta) 
&=\sin\?\theta\,u_\theta .
\end{aligned}
\end{equation}
Combining \eqref{eq:starsinwedge2} with \eqref{eq:d0}, \eqref{eq:sharp}, and \eqref{eq:star1} yields
\begin{equation}
  \label{eq:Qb}
  \begin{aligned}
    \sharp\, d\star \sin\?\theta\,d\theta\!\wedge\flat (u_\phi\myvec e_\phi  +u_\theta \myvec e_\theta)
&=-\tppp{u_\phi}{\phi}\myvec e_\phi-\sin^2\!\theta\,\tppp{u_\phi}{\theta}\,\myvec e_\theta -2\sin\?\theta\cos\?\theta\,u_\phi \,\myvec e_\theta
\\
   -\sharp \star d\star \sin\?\theta\,d\theta\!\wedge\star\,\flat (u_\phi\myvec e_\phi  +u_\theta \myvec e_\theta)
&=\tppp{u_\theta}{\theta}\,\myvec e_\phi+\tfrac{\cos\?\theta}{\sin\?\theta}\,u_\theta\,\myvec e_\phi-\tppp{u_\theta}{\phi}\,\myvec e_\theta.
  \end{aligned}
\end{equation}
Also, by \eqref{eq:flat}, \eqref{eq:star1}, \eqref{eq:d1}, and \eqref{eq:star20},
\begin{equation}
\begin{aligned}
  \label{eq:stardstarflat}
  \star\, d\star\flat (u_\phi\myvec e_\phi  +u_\theta \myvec e_\theta)&=\tppp{u_\phi}{\phi}+\tppp{u_\theta}{\theta}+\tfrac{\cos\?\theta}{\sin\?\theta}\,u_\theta
\\
\star\,d \,\flat (u_\phi\myvec e_\phi  +u_\theta \myvec e_\theta)
&= -\sin\?\theta\,\tppp{u_\phi}{\theta}-2\cos\?\theta\,u_\phi+\tfrac1{\sin\?\theta}\tppp{u_\theta}{\phi}
\end{aligned}
\end{equation}
Combining \eqref{eq:stardstarflat} with \eqref{eq:sharp} and  \eqref{eq:star1} yields
\begin{equation}
  \label{eq:Qc}
  \begin{aligned}
    -\sharp\star\sin\?\theta\,d\theta\,\star d\star \flat (u_\phi\myvec e_\phi  +u_\theta \myvec e_\theta)
&=(\tppp{u_\phi}{\phi}+\tppp{u_\theta}{\theta}+\tfrac{\cos\?\theta}{\sin\?\theta}\,u_\theta)\myvec e_\phi
\\
\sharp\,\sin\?\theta\,d\theta\star d\,\flat (u_\phi\myvec e_\phi  +u_\theta \myvec e_\theta)
&=( -\sin^2\!\theta\,\tppp{u_\phi}{\theta}-2\sin\?\theta\cos\?\theta\,u_\phi+\tppp{u_\theta}{\phi})\myvec e_\theta .
  \end{aligned}
\end{equation}
Combining \eqref{eq:M1L2b}, \eqref{eq:Q}, \eqref{eq:Qb}, and \eqref{eq:Qc} shows that
\begin{equation}
  \label{eq:M1L2final}
\begin{aligned}
  \relax [M_1,L_2]&\lp\begin{smallmatrix}u_\phi\myvec e_\phi  +u_\theta \myvec e_\theta\\h\end{smallmatrix}\rp
\\&=\mu\lp\begin{smallmatrix} \lp 2\ppp{u_\theta}{\theta}+2\tfrac{\cos\?\theta}{\sin\?\theta}u_\theta\rp\myvec e_\phi
-\lp 2\sin^2\!\theta\ppp{u_\phi}{\theta}+4\sin\?\theta\cos\?\theta u_\phi\rp\myvec e_\theta\\0 
\end{smallmatrix}\rp
\end{aligned}
\end{equation}
It is easier to calculate the commutators of $M_2$ using the second form of that operator:
\begin{equation}
  \label{eq:M2L1}
\begin{aligned}
 \relax [M_2,L_1]&=-\lp\begin{smallmatrix} \mu^2(\frac12-\sin^2\!\theta)&0&2\mu\\0&\mu^2(\frac12-\sin^2\!\theta)&0
\\2\mu\sin^2\!\theta&0&\mu^2(\frac12-\sin^2\!\theta)\end{smallmatrix}\rp
\lp\begin{smallmatrix} 0&0&\tfr{\sin^2\!\theta}\prt_\phi\\0&0&\prt_\theta\\\prt_\phi&\frac1{\sin\?\theta}\prt_\theta\sin\?\theta&0\end{smallmatrix}\rp
\\&\quad\;+ \lp\begin{smallmatrix} 0&0&\tfr{\sin^2\!\theta}\prt_\phi\\0&0&\prt_\theta\\
\prt_\phi&\frac1{\sin\?\theta}\prt_\theta\sin\?\theta&0\end{smallmatrix}\rp 
\lp\begin{smallmatrix} \mu^2(\frac12-\sin^2\!\theta)&0&2\mu\\0&\mu^2(\frac12-\sin^2\!\theta)&0\\
2\mu\sin^2\!\theta&0&\mu^2(\frac12-\sin^2\!\theta)\end{smallmatrix}\rp
\\&=
\lp\begin{smallmatrix} 0&-\frac{2\mu}{\sin\?\theta}\prt_\theta \sin\?\theta&0\\2\mu\prt_\theta\sin^2\!\theta&0&-2\mu^2\sin\?\theta\cos\?\theta\\
0&-2\mu^2\sin\?\theta\cos\?\theta&0 \end{smallmatrix}\rp\end{aligned}
\end{equation}
and
\begin{equation}
  \label{eq:M2L2}
\begin{aligned}
\relax  [M_2,L_2]&=\lp\begin{smallmatrix}
    \mu^2(\frac12-\sin^2\!\theta)&0&2\mu\\0&\mu^2(\frac12-\sin^2\!\theta)&0\\2\mu\sin^2\!\theta&0&\mu^2(\frac12-\sin^2\!\theta) \end{smallmatrix}\rp 
\lp\begin{smallmatrix}  \phantom{-}0&\mu\tfrac{\cos\?\theta}{\sin\?\theta}&0\\-\mu\sin\?\theta\cos\?\theta &0&0\\\phantom{-}0&0&0\end{smallmatrix}\rp
\\&\qquad-\lp\begin{smallmatrix}  \phantom{-}0&\mu\tfrac{\cos\?\theta}{\sin\?\theta}&0\\-\mu\sin\?\theta\cos\?\theta &0&0\\\phantom{-}0&0&0\end{smallmatrix}\rp 
\lp\begin{smallmatrix}  \mu^2(\frac12-\sin^2\!\theta)&0&2\mu\\0&\mu^2(\frac12-\sin^2\!\theta)&0\\
2\mu\sin^2\!\theta&0&\mu^2(\frac12-\sin^2\!\theta)\end{smallmatrix}\rp
\\&=\lp\begin{smallmatrix} 0&0&0\\0&0&2\mu^2\sin\?\theta\cos\?\theta\\0&2\mu^2\sin\?\theta\cos\?\theta&0 \end{smallmatrix}\rp
\end{aligned}
\end{equation}
Adding \eqref{eq:M2L1} and \eqref{eq:M2L2} yields
\begin{equation*}
  [M_2,L]=\lp\begin{smallmatrix}  0&-\frac{2\mu}{\sin\/\theta}\prt_\theta\sin\?\theta&0\\2\mu\prt_\theta\sin^2\!\theta&0&0\\0&0&0\end{smallmatrix}\rp,
\end{equation*}
which implies that
\begin{equation}
  \label{eq:M2Lalt}
  [M_2,L]\lp\begin{smallmatrix} u_\phi\myvec e_\phi  +u_\theta \myvec e_\theta\\h \end{smallmatrix}\rp
=\mu\lp\begin{smallmatrix} -(2\ppp{u_\theta}{\theta}+2\frac{\cos\?\theta}{\sin\?\theta}\,u_\theta)\myvec e_\phi
+(2\sin^2\!\theta\ppp{u_\phi}{\theta}+4\sin\?\theta\cos\?\theta \,u_\phi) \myvec e_\theta\\0\end{smallmatrix}\rp.
\end{equation}
Combining \eqref{eq:M1L1}, \eqref{eq:M1L2final}, and \eqref{eq:M2Lalt} yields~\eqref{eq:LMML}.
\end{proof}

\section{Uniform Bound} 
\label{sec:unif}

\begin{proof}[Proof of the \cref{thm:unif}]
  In similar fashion to the equatorial beta-plane model, the equations \eqref{eq:euler} or \eqref{eq:euler-local} can be symmetrized by making the
  change of variables \cite[\S 2]{MR2487851}
$h=\tfrac{\lp 1+{\del \tilde h}/2\rp^2-1}\del$
and dividing the resulting equation for $\tilde h$ by $1+\frac{\del \tilde h}2$. After dropping the tilde, this yields
\begin{equation}
  \label{eq:eulermod}
\begin{aligned}
 \prt_t &\myvec u+\grad_{\?\myvec u}\myvec u+\tfrac h2\grad h=-\tfr\del \grad h+\tfrac{z}{\eps}\myvec u^\perp,
\\
\prt_t &h+\grad_{\?\myvec u}  h+\tfrac h2\div \myvec u=-\tfr\del\div \myvec u
\end{aligned}
\end{equation}
in place of \eqref{eq:euler}. In view of \eqref{eq:covderiv} and \eqref{eq:divstargrad}, \eqref{eq:eulermod} is symmetric hyperbolic with respect to the inner
product~\eqref{eq:ip}. Hence, by \cite[\S 16.1, Theorem 1.2, Proposition 1.3, and Proposition 1.4]{MR4703941}, which as noted near the beginning of
that section remain valid when the domain is a compact Riemannian manifold, the initial-value problem \eqref{eq:euler} or \eqref{eq:euler-local} with
initial data~$\lp\begin{smallmatrix}\myvec u_0\\h_0\end{smallmatrix}\rp$ satisfying \eqref{eq:initbnd} has a unique solution in $C^0([0,T];H^4)$ for
some positive $T$ that might depend on $\del$ and $\eps$ as well as on the bounds $C$ and $K$ in the assumptions of the theorem, and
\eqref{eq:unifbnd} holds with $T(C,K)$ replaced by that $T$. Hence there remains to show that $T$ can be taken independent of $\del$ and $\eps$.

Taking the inner product of $2\lp\begin{smallmatrix}\myvec u\\h\end{smallmatrix}\rp$ with equation \eqref{eq:eulermod} and using the
definition~\eqref{eq:L}\textendash\eqref{eq:L2} of $L$ and \eqref{eq:Lanti} yields
\begin{equation}
  \label{eq:forL2bnd}
\begin{aligned}
  \ddt \left\|\lp\begin{smallmatrix}\myvec u\\h\end{smallmatrix}\rp\right\|_{L^2}^2
&=-2\ip{\lp\begin{smallmatrix}\myvec u\\h\end{smallmatrix}\rp}{\lp\begin{smallmatrix}\grad_{\?\myvec u}\myvec u\\\grad_{\?\myvec u}h\end{smallmatrix}\rp}
-2\ip{\lp\begin{smallmatrix}\myvec u\\h\end{smallmatrix}\rp}{\lp\begin{smallmatrix}\frac h2\grad h\\\frac h2 \div u\end{smallmatrix}\rp}
+\tfrac{2}{\del}\ip{\lp\begin{smallmatrix}\myvec u\\h\end{smallmatrix}\rp}{L\lp\begin{smallmatrix}\myvec u\\h\end{smallmatrix}\rp}
\\&=-2\ip{\lp\begin{smallmatrix}\myvec u\\h\end{smallmatrix}\rp}{\lp\begin{smallmatrix}\grad_{\?\myvec u}\myvec u\\\grad_{\?\myvec u}h\end{smallmatrix}\rp}
-2\ip{\lp\begin{smallmatrix}\myvec u\\h\end{smallmatrix}\rp}{\lp\begin{smallmatrix}\frac h2\grad h\\\frac h2 \div u\end{smallmatrix}\rp}.
\end{aligned}
\end{equation}
Using the formulas~\eqref{eq:vgradf}, \eqref{eq:covderiv}, \eqref{eq:div} and~\eqref{eq:grad} for $\grad_{\?\myvec u}$, $\div u$, and $\grad h$,  the
definition~\eqref{eq:ip} of the inner product, and integration by parts therefore shows that
\begin{equation}
  \label{eq:L2bnd}
\ddt \left\|\lp\begin{smallmatrix}\myvec u\\h\end{smallmatrix}\rp\right\|_{L^2}^2
\le c \|\myvec u\|_{C^1}\left\|\lp\begin{smallmatrix}\myvec u\\h\end{smallmatrix}\rp\right\|_{L^2}^2.
\end{equation}
where here and later $c$ denotes a constant, which may be different in different occurrences, that depends on $\del$ and $\eps$ only through smooth
dependence on their ratio~$\mu$ that appears in the operator~$M$. The constants $c_j$ and $k$ will denote constants that also depend on $\del$ and
$\eps$ only through smooth dependence on~$\mu$, but which remain the same in all occurrences. It would not be very hard to obtain bounds for the form
$\widetilde c(1+\mu^p)$ for how these constants depend on $\mu$, which would yield bounds of that form for the constants~$T$ and $B$ in the statement
of the theorem. However, since such explicit bounds are not required for the theorem, and obtaining them requires longer calculations, we will just
note that the facts that $M$ contains only positive powers of $\mu$ and that all the large terms will drop out of the estimates will imply that all those
constants will be bounded for bounded values of $\mu$.

Next, letting $p$ be a positive integer, applying $M^p$ to \eqref{eq:eulermod}, writing $M^p$ applied an operator applied to
$\lp\begin{smallmatrix}\myvec u\\h\end{smallmatrix}\rp$ as that operator applied to
$M^p\lp\begin{smallmatrix}\myvec u\\h\end{smallmatrix}\rp$ plus the commutator of $M^p$ with that operator applied to
$\lp\begin{smallmatrix}\myvec u\\h\end{smallmatrix}\rp$, and using \eqref{eq:LMML} repeatedly yields
\begin{equation}
  \label{eq:forMbnd}
  \begin{aligned}
   \ddt M^p\lp\begin{smallmatrix}\myvec u\\h\end{smallmatrix}\rp
&=-M^p\lp\begin{smallmatrix}\grad_{\?\myvec u}&0\\0&\grad_{\?\myvec u}\end{smallmatrix}\rp
\lp\begin{smallmatrix}\myvec u\\ h\end{smallmatrix}\rp
+\tfrac1\del M^pL\lp\begin{smallmatrix}\myvec u\\h\end{smallmatrix}\rp
\\&=-\lp\begin{smallmatrix}\grad_{\?\myvec u}&0\\0&\grad_{\?\myvec u}\end{smallmatrix}\rp
M^p\lp\begin{smallmatrix}\myvec u\\ h\end{smallmatrix}\rp
-\left[M^p,\lp\begin{smallmatrix}\grad_{\?\myvec u}&0\\0&\grad_{\?\myvec u}\end{smallmatrix}\rp\right]\lp\begin{smallmatrix}\myvec u\\h\end{smallmatrix}\rp
+\tfrac1\del LM^p\lp\begin{smallmatrix}\myvec u\\h\end{smallmatrix}\rp.
  \end{aligned}
\end{equation}
Taking the inner product of $2M^p\lp\begin{smallmatrix}\myvec u\\h\end{smallmatrix}\rp$ with equation~\eqref{eq:forMbnd} and continuing in similar
fashion to the calculation of the time derivative of the $L^2$ norm in \eqref{eq:forL2bnd}\textendash\eqref{eq:L2bnd} yields
\begin{equation}
  \label{eq:Mbnd}
  \ddt\left\|M^p\lp\begin{smallmatrix}\myvec u\\h\end{smallmatrix}\rp\right\|_{L^2}^2\le 
c\lp \|\myvec u\|_{C^1}\left\|M^p\lp\begin{smallmatrix}\myvec u\\h\end{smallmatrix}\rp\right\|_{L^2}^2
+\left\|M^p\lp\begin{smallmatrix}\myvec u\\h\end{smallmatrix}\rp\right\|_{L^2} 
\left\| \left[M^p,\lp\begin{smallmatrix}\grad_{\?\myvec u}&0\\0&\grad_{\?\myvec u}\end{smallmatrix}\rp\right]
\lp\begin{smallmatrix}\myvec u\\h\end{smallmatrix}\rp \right\|_{L^2}\rp.
\end{equation}
Adding a constant times \eqref{eq:L2bnd} to \eqref{eq:Mbnd} with $p=2$ 
yields
\begin{equation}
  \label{eq:h4est}
\begin{aligned}
  \ddt &\lb k \left\|\lp\begin{smallmatrix}\myvec u\\h\end{smallmatrix}\rp\right\|_{L^2}^2
+\left\|M^2\lp\begin{smallmatrix}\myvec u\\h\end{smallmatrix}\rp\right\|_{L^2}^2\rb
\\&\le c\lp \|\myvec u\|_{C^1}\lb k \left\|\lp\begin{smallmatrix}\myvec u\\h\end{smallmatrix}\rp\right\|_{L^2}^2 
+ \left\|M^2\lp\begin{smallmatrix}\myvec u\\h\end{smallmatrix}\rp\right\|_{L^2}^2\rb 
+\left\|M^2\lp\begin{smallmatrix}\myvec u\\h\end{smallmatrix}\rp\right\|_{L^2} 
\left\| \left[M^2,\lp\begin{smallmatrix}\grad_{\?\myvec u}&0\\0&\grad_{\?\myvec u}\end{smallmatrix}\rp\right]
\lp\begin{smallmatrix}\myvec u\\h\end{smallmatrix}\rp \right\|_{L^2}\rp.
\end{aligned}
\end{equation}
Now use the standard Sobolev estimate (e.g. \cite[p. 3893]{MR3478875})
\begin{equation}\label{eq:c1est}
  \|\myvec u\|_{C^1}\le c\|\myvec u\|_{H^4}
\end{equation}
and the commutator estimate \cite[(2.1.12)]{MR3478875}
\begin{equation}
  \label{eq:commest}
  \|[ (-\Delta)^{\!\frac s2},\grad_{\?\myvec u}]\myvec u\|\le c\|\myvec u\|_{C^1}\|\myvec u\|_{H^s}
\end{equation}
where the mollifier in~\cite{MR3478875} has been omitted
since the estimate remains valid without it, and the estimate has been written for vector fields rather than for one-forms in light of the
equivalence~\eqref{eq:lap=hodge:u}. The definition~\eqref{eq:M}\textendash\eqref{eq:M2} of $M$ implies that
\begin{equation}\label{eq:M2H4}
M^2-\lp\begin{smallmatrix}\Delta^{\?2}&0\\0&\Delta^{\?2}\end{smallmatrix}\rp
\text{ is a non-diagonal second-order operator.}
\end{equation}
Since the derivation of~\eqref{eq:commest} in \cite{MR3478875} remains valid when the commutator is applied to a
function rather than a one-form or vector field, \eqref{eq:commest} implies that
\begin{equation}
  \label{eq:comm4}
    \left\| \left[\lp\begin{smallmatrix}\Delta^{\?2}&0\\0&\Delta^{\?2}\end{smallmatrix}\rp,
\lp\begin{smallmatrix}\grad_{\?\myvec u}&0\\0&\grad_{\?\myvec u}\end{smallmatrix}\rp\right]
\lp\begin{smallmatrix}\myvec u\\h\end{smallmatrix}\rp \right\|_{L^2}
\le c\|\myvec u\|_{C^1}\left\|\lp\begin{smallmatrix}\myvec u\\h\end{smallmatrix}\rp \right\|_{H^4}.
\end{equation}
By \eqref{eq:M2H4}, the remaining terms in $M^2$ yield terms in
$\left[M^2,\lp\begin{smallmatrix}\grad_{\?\myvec u}&0\\0&\grad_{\?\myvec u}\end{smallmatrix}\rp\right]$ in which
at most three derivatives are applied to two factors of the dependent variables. Hence one of those factors contains at most one derivative,
so pulling it out in sup norm from the $L^2$ norm of the commutator yields the estimate
\begin{equation}
  \label{eq:comm3}
\left\|\text{remaining terms of $\left[M^2,\lp\begin{smallmatrix}\grad_{\?\myvec u}&0\\0&\grad_{\?\myvec u}\end{smallmatrix}\rp\right]
\lp\begin{smallmatrix}\myvec u\\h\end{smallmatrix}\rp$}\right\|_{L^2}
\le  c\left\|\lp\begin{smallmatrix}\myvec u\\h\end{smallmatrix}\rp\right\|_{C^1}
\left\|\lp\begin{smallmatrix}\myvec u\\h\end{smallmatrix}\rp\right\|_{H^3}.
\end{equation}
Combining \eqref{eq:comm4}\textendash\eqref{eq:comm3} and using \eqref{eq:c1est} shows that
\begin{equation}
  \label{eq:commest2}
  \left\| \left[M^2,\lp\begin{smallmatrix}\grad_{\?\myvec u}&0\\0&\grad_{\?\myvec u}\end{smallmatrix}\rp\right]
\lp\begin{smallmatrix}\myvec u\\h\end{smallmatrix}\rp \right\|_{L^2}
\le c \left\|\lp\begin{smallmatrix}\myvec u\\h\end{smallmatrix}\rp\right\|_{H^4}^2.
\end{equation}
Moreover, in view of \eqref{eq:M2H4} and the Sobolev estimate
$\left\|\lp\begin{smallmatrix}\myvec u\\h\end{smallmatrix}\rp\right\|_{H^2}^2\le
c \left\|\lp\begin{smallmatrix}\myvec u\\h\end{smallmatrix}\rp\right\|_{H^4} \left\|\lp\begin{smallmatrix}\myvec u\\h\end{smallmatrix}\rp\right\|_{L^2}$,
for $k$ sufficiently large 
the norm $[k\|\cdot\|_{L^2}^2+\|M^2\cdot\|_{L^2}^2]^{\frac12}$ is equivalent to the
$H^4$ norm, i.e.,
\begin{equation}\label{eq:equivnorm}
c_1\|\lp\begin{smallmatrix}\myvec u\\h\end{smallmatrix}\rp\|_{H^4}\le 
[k\left\|\lp\begin{smallmatrix}\myvec u\\h\end{smallmatrix}\rp\right\|_{L^2}^2
+\left\|M^2\lp\begin{smallmatrix}\myvec u\\h\end{smallmatrix}\rp\right\|_{L^2}^2]^{\frac12}
\le c_2\|\lp\begin{smallmatrix}\myvec u\\h\end{smallmatrix}\rp\|_{H^4}.
\end{equation}
Substituting \eqref{eq:c1est}, \eqref{eq:commest2}, and the left inequality of \eqref{eq:equivnorm} into \eqref{eq:h4est} 
with $k$ sufficiently large therefore yields
\begin{equation}
  \label{eq:h4est2}
  \ddt \lb k \left\|\lp\begin{smallmatrix}\myvec u\\h\end{smallmatrix}\rp\right\|_{L^2}^2
+\left\|M^2\lp\begin{smallmatrix}\myvec u\\h\end{smallmatrix}\rp\right\|_{L^2}^2\rb
\le c_3\lb k \left\|\lp\begin{smallmatrix}\myvec u\\h\end{smallmatrix}\rp\right\|_{L^2}^2
+\left\|M^2\lp\begin{smallmatrix}\myvec u\\h\end{smallmatrix}\rp\right\|_{L^2}^2\rb^{\frac32}.
\end{equation}
Integrating the differential inequality~\eqref{eq:h4est2} with the initial bound
\begin{equation*}
\lb k \left\|\lp\begin{smallmatrix}\myvec u\\h\end{smallmatrix}\rp\right\|_{L^2}^2
+\left\|M^2\lp\begin{smallmatrix}\myvec u\\h\end{smallmatrix}\rp\right\|_{L^2}^2\rb_{t=0}\le c_2K
\end{equation*}
obtained from~\eqref{eq:initbnd} and the right inequality in~\eqref{eq:equivnorm}
yields a uniform bound for $\lb k \left\|\lp\begin{smallmatrix}\myvec u\\h\end{smallmatrix}\rp\right\|_{L^2}^2
+\left\|M^2\lp\begin{smallmatrix}\myvec u\\h\end{smallmatrix}\rp\right\|_{L^2}^2\rb$ on a uniform time interval. Using the left inequality of 
\eqref{eq:equivnorm} once more then yields~\eqref{eq:unifbnd}.
\end{proof}

\section{Convergence}
\label{sec:conv}

\begin{proof}[Proof of  \cref{thm:converge} and \cref{thm:threescale}]
  For \cref{thm:converge} the assumption that $\mu$ is fixed implies that \eqref{eq:mucond} holds with $C\eqdef\mu$.  For~\cref{thm:threescale} the
  definitions \eqref{eq:mu} and \eqref{eq:nu} imply that $\mu=\eps\nu$, so \eqref{eq:nuconst} implies that \eqref{eq:mucond} holds with $C=\nu$ for
  $\eps\le1$. In both cases, \eqref{eq:uinit3}\textendash\eqref{eq:unit3bnd} imply that \eqref{eq:initbnd} holds with $K$ replaced by $3K$ for
  $\eps\le1$. Hence for $\eps\le1$ the conclusions of \cref{thm:unif}, and in particular~\eqref{eq:unifbnd}, hold with the above values of $C$ and
  $K$.

For \cref{thm:converge}, assumption~\eqref{eq:wellprep0} together with the definition~\eqref{eq:mu} of $\mu$ ensure that the terms on the right side of
\eqref{eq:eulermod} vanish at time zero, which implies that the initial value of $\prt_t\! \lp\begin{smallmatrix}\myvec u\\h\end{smallmatrix}\rp$
calculated from that system is uniformly bounded. For \cref{thm:threescale} with the scaling \eqref{eq:nuconst} the terms of order $\tfr{\eps^2}$ in
\eqref{eq:eulermod} vanish at time zero iff 
\begin{equation}\label{eq:initepsm2}
\grad h_{0,0}=0\quad\text{and}\quad \div\myvec u_{0,0}=0,
\end{equation}
while the terms of order $\tfr{\eps}$ vanish iff
\begin{equation}\label{eq:initepsm1}
-\grad h_{0,1}+\nu \myvec u_{0,0}^\perp=0\quad\text{and}\quad \div \myvec u_{0,1}=0.
\end{equation}
The second equation of \eqref{eq:initepsm2} and the first equation of \eqref{eq:initepsm1} have the form \eqref{eq:wellprep0} with 
$h_{0,0}$ replaced by $h_{0,1}$ and $\mu$ replaced by $\nu$. \cref{lem:N(L)} together with the remaining conditions of
\eqref{eq:initepsm2}\textendash\eqref{eq:initepsm1} therefore imply that for the case of \cref{thm:threescale} the initial value of
$\prt_t\! \lp\begin{smallmatrix}\myvec u\\h\end{smallmatrix}\rp$ calculated from \eqref{eq:eulermod} will be uniformly bounded iff
assumption~\eqref{eq:unit3conds} of that theorem holds.  

Since the initial value of $\prt_t\!\lp\begin{smallmatrix}\myvec u\\h\end{smallmatrix}\rp$ is therefore uniformly bounded for both theorems, we now
proceed to obtain energy estimates for those time derivatives. Differentiating~\eqref{eq:eulermod} with respect to $t$ yields
\begin{equation}
  \label{eq:dteulermod}
\begin{aligned}
 \prt_t &\myvec u_t+\grad_{\?\myvec u}\myvec u_t+\grad_{\?\myvec u_t}\myvec u
+\tfrac h2\grad h_t+\tfrac {h_t}2\grad h=-\tfr\del \grad h_t+\tfrac{z}{\eps}\myvec u_t^\perp,
\\
\prt_t &h_t+\grad_{\?\myvec u}  h_t+\grad_{\?\myvec u_t}  h+\tfrac h2\div \myvec u_t+\tfrac {h_t}2\div \myvec u=-\tfr\del\div \myvec u_t,
\end{aligned}
\end{equation}
where as usual a subscript $t$ denotes that the variable has been differentiated with respect to $t$. Taking the inner product of 
$2 \lp\begin{smallmatrix}\myvec u_t\\h_t\end{smallmatrix}\rp$ with the equation \eqref{eq:dteulermod} yields, in similar fashion to the derivation of
\eqref{eq:forL2bnd}\textendash\eqref{eq:L2bnd}
\begin{equation}
  \label{eq:dtL2bnd}
\begin{aligned}
  \ddt \left\|\lp\begin{smallmatrix}\myvec u_t\\h_t\end{smallmatrix}\rp\right\|_{L^2}^2
&=-2\ip{\lp\begin{smallmatrix}\myvec u_t\\h_t\end{smallmatrix}\rp}{\lp\begin{smallmatrix}\grad_{\?\myvec u}\myvec u_t\\\grad_{\?\myvec u}h_t\end{smallmatrix}\rp}
-2\ip{\lp\begin{smallmatrix}\myvec u_t\\h_t\end{smallmatrix}\rp}{\lp\begin{smallmatrix}\grad_{\?\myvec u_t}\myvec u\\\grad_{\?\myvec u_t}h\end{smallmatrix}\rp}
\\&\quad
-2\ip{\lp\begin{smallmatrix}\myvec u_t\\h_t\end{smallmatrix}\rp}{\lp\begin{smallmatrix}\frac h2\grad h_t\\\frac h2 \div u_t\end{smallmatrix}\rp}
-2\ip{\lp\begin{smallmatrix}\myvec u_t\\h_t\end{smallmatrix}\rp}{\lp\begin{smallmatrix}\frac {h_t}2\grad h\\\frac {h_t}2 \div u\end{smallmatrix}\rp}
+\tfrac{2}{\del}\ip{\lp\begin{smallmatrix}\myvec u_t\\h_t\end{smallmatrix}\rp}{L\lp\begin{smallmatrix}\myvec u_t\\h_t\end{smallmatrix}\rp}
\\&=
-2\ip{\lp\begin{smallmatrix}\myvec u_t\\h_t\end{smallmatrix}\rp}{\lp\begin{smallmatrix}\grad_{\?\myvec u}\myvec u_t\\\grad_{\?\myvec u}h_t\end{smallmatrix}\rp}
-2\ip{\lp\begin{smallmatrix}\myvec u_t\\h_t\end{smallmatrix}\rp}{\lp\begin{smallmatrix}\grad_{\?\myvec u_t}\myvec u\\\grad_{\?\myvec u_t}h\end{smallmatrix}\rp}
\\&\quad
-2\ip{\lp\begin{smallmatrix}\myvec u_t\\h_t\end{smallmatrix}\rp}{\lp\begin{smallmatrix}\frac h2\grad h_t\\\frac h2 \div u_t\end{smallmatrix}\rp}
-2\ip{\lp\begin{smallmatrix}\myvec u_t\\h_t\end{smallmatrix}\rp}{\lp\begin{smallmatrix}\frac {h_t}2\grad h\\\frac {h_t}2 \div u\end{smallmatrix}\rp}
\\&\le
c \|\myvec u\|_{C^1}\left\|\lp\begin{smallmatrix}\myvec u_t\\h_t\end{smallmatrix}\rp\right\|_{L^2}^2.
\end{aligned}
\end{equation}
since the new terms involving $\grad_{\myvec u_t}$ or $\frac{h_t}2$ can be estimated directly by that bound. 
In view of \eqref{eq:c1est}, \eqref{eq:unifbnd}, and the uniform bound on the initial value of 
$\left\|\lp\begin{smallmatrix}\myvec u_t\\h_t\end{smallmatrix}\rp\right\|_{L^2}$ obtained above, \eqref{eq:dtL2bnd} implies that
\begin{equation}
  \label{eq:intdtL2bnd}
  \left\|\lp\begin{smallmatrix}\myvec u_t\\h_t\end{smallmatrix}\rp\right\|_{L^2}^2
\le c e^{c t}
\end{equation}
on the uniform time interval provided by \cref{thm:unif}. 

Let $T$ equal $T(\mu,3K)$ for \cref{thm:converge} or $T(\nu,3K)$ for \cref{thm:threescale}. As in the classical theory of singular limits, the uniform
estimates~\eqref{eq:unifbnd} and \eqref{eq:intdtL2bnd} imply by Ascoli's theorem plus the weak-$*$ compactness of $L^\infty([0,T];H^4)$ that for every
sequence of values of $\del,\eps$ tending to zero while satisfying \eqref{eq:muconst} or \eqref{eq:nuconst} a subsequence of
$\lp\begin{smallmatrix}\myvec u\\h\end{smallmatrix}\rp$ converges weak-* in $L^\infty([0,T];H^4)$ and strongly in $C^0([0,T];L^2)$. By the standard
interpolation inequality $\|\cdot\|_{H^r}\le C(r,s)\|\cdot\|_{H^s}^{\frac rs}\|\cdot\|_{L^2}^{1-\frac rs}$, the $C^0([0,T];L^2)$ convergence together
with the uniform~$C^0([0,T];H^4)$ bound~\eqref{eq:unifbnd} yield convergence also in $C^0([0,T];H^r)$ for $r<4$. Define
$\lp\begin{smallmatrix}\overline{\myvec u}\\\overline{h}\end{smallmatrix}\rp\eqdef \lim \lp\begin{smallmatrix}\myvec u\\h\end{smallmatrix}\rp$,
where for now the limit is taken along the above subsequence.

Multiplying~\eqref{eq:eulermod} by $\del$ yields
\begin{equation}
  \label{eq:deleulermod}
\begin{aligned}
\del\lb\prt_t \lp\begin{smallmatrix}\myvec u\\h\end{smallmatrix}\rp 
+ \lp\begin{smallmatrix}\grad_{\?\myvec u}\myvec u\\ \grad_{\?\myvec u}h\end{smallmatrix}\rp
+\tfrac h2 \lp\begin{smallmatrix}\grad h\\ \div \myvec u\end{smallmatrix}\rp\rb
=\lp\begin{smallmatrix} -\grad h+\mu z\myvec u^\perp\\ -\div \myvec u\end{smallmatrix}\rp
=L\lp\begin{smallmatrix}\myvec u\\h\end{smallmatrix}\rp.
\end{aligned}  
\end{equation}
By the above bounds, taking the limit of \eqref{eq:deleulermod} yields
$L\lp\begin{smallmatrix}\overline{\myvec u}\\\overline{h}\end{smallmatrix}\rp=0$ if $\mu$ is fixed, or
$L_1\lp\begin{smallmatrix}\overline{\myvec u}\\\overline{h}\end{smallmatrix}\rp=0$ if $\nu$ is fixed, since in the latter case $\mu$ also tends to
zero. In the former case the equation~$L\lp\begin{smallmatrix}\overline{\myvec u}\\\overline{h}\end{smallmatrix}\rp=0$ implies by~\cref{lem:N(L)} that
the limit~$\lp\begin{smallmatrix}\overline{\myvec u}\\\overline{h}\end{smallmatrix}\rp$ has the form
\begin{equation}
  \label{eq:lim1}
  \overline{\myvec u}=\overline{u}_\phi(t,\theta)\myvec e_\phi\qquad\text{and}\qquad 
\overline h=\overline h(t,\theta)\eqdef c(t)-\mu\int_0^\theta \cos\?\beta\,\sin\?\beta\, \overline{u}_\phi(t,\beta)\,d\beta.
\end{equation}
In the latter case the equation~$L_1\lp\begin{smallmatrix}\overline{\myvec u}\\\overline{h}\end{smallmatrix}\rp=0$ implies that $\overline h$ is
independent of both $\phi$ and $\theta$ and that $\div\overline{\myvec u}=0$.  One might expect that the next step for this latter case would be to
apply $P(0)^*$ to~\eqref{eq:eulermod}, which by the first equation of~\eqref{eq:P0L10} eliminates the term of size $O(\frac1\del)$, so that
multiplying the result by $\eps$ and taking the limit yields another equation that
$\lp\begin{smallmatrix}\overline{\myvec u}\\\overline{h}\end{smallmatrix}\rp$ satisfies. However, by the second equation of~\eqref{eq:P0L10} this
only yields $\Avphi\overline u_\theta=0$.  Hence we will instead apply $\perp$ and then $\div$ to the first equation in~\eqref{eq:eulermod}, which by
the first equation in~\eqref{eq:divgradperp0} also eliminates the term of size $O(\frac1\del)$ there. Hence, in view of the second equation
in~\eqref{eq:divgradperp0} and formula~\eqref{eq:div} for the divergence, multiplying the result by $\eps$ and taking the limit yields
$0=\div (\cos\?\theta\, \overline{\myvec u})=\cos\?\theta \div\overline{\myvec u}-\sin\?\theta\,\overline u_\theta$. Since we already saw that
$\div\overline{\myvec u}=0$ in this latter case, $\overline u_\theta=0$, and since $\div\overline{\myvec u}=0$ we then obtain that $\overline u_\phi$
is independent of $\phi$. Thus, in the case in which $\nu$ is fixed we obtain \eqref{eq:lim1} with $\mu$ set equal to its limit value zero.

By \eqref{eq:PL0}, applying $P(\mu)^*$ to \eqref{eq:eulermod} yields
\begin{equation}
  \label{eq:Pem}
\begin{aligned}
\prt_t P(\mu)^*       \lp\begin{smallmatrix}\myvec u\\h\end{smallmatrix}\rp 
+ P(\mu)^*\lb\lp\begin{smallmatrix}\grad_{\?\myvec u}\myvec u\\ \grad_{\?\myvec u}h\end{smallmatrix}\rp
+\tfrac h2 \lp\begin{smallmatrix}\grad h\\ \div \myvec u\end{smallmatrix}\rp\rb=\lp\begin{smallmatrix}\myvec 0\\0\end{smallmatrix}\rp.
\end{aligned}
\end{equation}
Since the large terms have been eliminated, the convergence obtained above suffices to take the limit of~\eqref{eq:Pem} as $\del,\eps\to0$, which yields
\begin{equation}
  \label{eq:limPem}
\begin{aligned}
\prt_t P(\mu)^*       \lp\begin{smallmatrix}\overline{\myvec u}\\\overline{h}\end{smallmatrix}\rp 
+ P(\mu)^*\lb\lp\begin{smallmatrix}\grad_{\?\overline{\myvec u}}\overline{\myvec u}\\ \grad_{\?\overline{\myvec u}}\overline{h}\end{smallmatrix}\rp
+\tfrac {\overline h}2 \lp\begin{smallmatrix}\grad \overline{h}\\ \div \overline{\myvec u}\end{smallmatrix}\rp\rb
=\lp\begin{smallmatrix}\myvec 0\\0\end{smallmatrix}\rp,
\end{aligned}
\end{equation}
where $\mu$ is replaced by zero in the case when $\nu$ rather than $\mu$ is held fixed. Moreover, \eqref{eq:lim1} together with the formulas
\eqref{eq:vgradf}, \eqref{eq:covderiv}, \eqref{eq:div}, and \eqref{eq:grad} for $\grad_{\?\overline{\myvec u}}$, $\div$ and $\grad$ imply that the only
nonzero terms in $\lb\lp\begin{smallmatrix}\grad_{\?\overline{\myvec u}}\overline{\myvec u}\\ \grad_{\?\overline{\myvec u}}\overline{h}\end{smallmatrix}\rp
+\tfrac {\overline h}2 \lp\begin{smallmatrix}\grad \overline{h}\\ \div \overline{\myvec u}\end{smallmatrix}\rp\rb$ are multiples of $\myvec e_\theta$,
which are eliminated by $P(\mu)^*$ or $P(0)^*$. Using \eqref{eq:lim1} together with the formula~\eqref{eq:Pstar} for $P(\mu)^*$ therefore shows that
\eqref{eq:limPem} reduces to
$\prt_t \lp\begin{smallmatrix}\overline{\myvec u}
-\mu \myvec e_\phi\frac{\cos\?\theta}{\sin^2\!\theta}\widehat I_\theta(\sin\?\theta \overline h)\\\overline h\end{smallmatrix}\rp
=\lp\begin{smallmatrix}\myvec 0\\0\end{smallmatrix}\rp$,
which simplifies to 
$\prt_t \lp\begin{smallmatrix}\overline{\myvec u}\\\overline h\end{smallmatrix}\rp
=\lp\begin{smallmatrix}\myvec 0\\0\end{smallmatrix}\rp$ even when $\mu$ is not replaced by zero. 
Together with~\eqref{eq:lim1}, this yields the conclusion of the theorems for the limit taken along the subsequence.
Since the limit is unique, convergence holds without restricting to subsequences.
\end{proof}

\appendix
\section{Review of Differential Geometry on the Sphere}

After choosing the length scale so that the radius of the sphere equals one, the sphere can be described by the spherical coordinates
\begin{equation}
  \label{eq:polar}
  x=\cos\?\phi\sin\?\theta, \quad
  y=\sin\?\phi\sin\?\theta, \quad
  z=\cos\?\theta,
\end{equation}
where $\phi$ is the longitude and $\theta$ the co-latitude, i.e., the angle from the North pole. In order to describe the velocity~$\myvec u$ of a
fluid moving on the sphere, and its derivatives, in invariant terms, we use the formalism of differential geometry.

\subsection{Local coordinates of vectors on the sphere}

A vector at a point on a general $d$-dimensional manifold is an element of the tangent space at that
point, which can be viewed as the set of equivalence classes of curves on the manifold that pass through the point and agree there to
first order. Since the coordinates of the sphere are $(\phi,\theta)$, we will let $\myvec e_{\phi}$ denote the vector in the direction of increasing longitude,
i.e.  the equivalence class of the curve~$(\phi_0+t,\theta_0)$, and let $\myvec e_{\theta}$ denote the vector in the direction of increasing
co-latitude, i.e.  the equivalence class of the curve~$(\phi_0,\theta_0+t)$. Although these vectors, like the coordinate~$\phi$, are not well defined
at the poles, that apparent difficulty will not cause any actual problem because those two points form a set of measure zero on the sphere. A general
vector on the sphere then has the form
\begin{equation}
  \label{eq:ucoords}
  \myvec u=u_{\phi}\myvec e_\phi+u_\theta \myvec e_\theta.
\end{equation}
Differentiating \eqref{eq:polar} shows that, considered as vectors in $\mathbb R^3$,
\begin{equation}
  \label{eq:ephi-etheta}
  \myvec e_\phi=\lp\begin{smallmatrix} -\sin\?\phi\sin\?\theta \\ \cos\?\phi\sin\?\theta\\0\end{smallmatrix}\rp,
\qquad \myvec e_\theta =\lp\begin{smallmatrix} \cos\?\phi\cos\?\theta\\ \sin\?\phi\cos\?\theta\\ -\sin\?\theta \end{smallmatrix}\rp.
\end{equation}

\subsection{Vector fields and covariant differentiation}
A vector field is an element of the tangent space obtained by choosing a tangent vector at each point of the
manifold. A vector field~$\myvec v$ acts on the space of functions defined on the manifold by $\grad_{\myvec v} f\eqdef \sum_j v_j\prt_{p_j}f$. 
For the sphere this becomes
\begin{equation}
  \label{eq:vgradf}
  \grad_{\?\myvec u} f=u_{\phi}\prt_\phi f+u_\theta \prt_\theta f. 
\end{equation}

When a manifold is embedded in Euclidean space then the inner product of two tangent vectors can be determined by
\begin{equation}
  \label{eq:veclen}
  g(\myvec u,\myvec v)\eqdef 
\lp \ddt X(p_1+u_1 t,\ldots p_d+u_dt)\eval{t=0}\rp\cdot
\lp \ddt X(p_1+v_1 t,\ldots p_d+v_dt)\eval{t=0}\rp
\end{equation}
where $X(p)$ denotes the Euclidean space point corresponding to a
point with local coordinates $\myvec p$, the dot product on the right side
is the Euclidean inner product, and $g$ denotes the inner product on the tangent space of the manifold.
For the coordinates $(\phi,\theta)$ on the sphere this yields, in view of \eqref{eq:ephi-etheta},
\begin{equation}\label{eq:g}
g(u_\phi \myvec e_\phi+u_\theta \myvec e_\theta, v_\phi \myvec e_\phi+v_\theta \myvec e_\theta)=\sin^2\!\theta \,u_\phi v_\phi+u_\theta v_\theta.
\end{equation}
In particular, the ``metric tensor'' is
\begin{equation}
  \label{eq:sphereip}
 \begin{pmatrix}
    g_{\phi,\phi}&g_{\phi,\theta}\\ g_{\theta,\phi}&g_{\theta,\theta}
  \end{pmatrix}
\eqdef
\begin{pmatrix}
  g(\myvec e_\phi,\myvec e_\phi)&g(\myvec e_\phi,\myvec e_\theta)\\
g(\myvec e_\theta,\myvec e_\phi)&g(\myvec e_\theta,\myvec e_\theta)
\end{pmatrix}
=
\begin{pmatrix}
  \sin^2\!\theta&0\\0&1
\end{pmatrix}.
\end{equation}

In order to differentiate a vector field it is necessary to take into
account not only the changes in its components but also the change of
direction of unit vectors. In local coordinates $\prt_{p_j}$ the
derivative of a vector field can be described in terms of the
Christoffel symbols $\Gamma^{k}_{ij}$, which are characterized by the
property that $\grad_{\myvec e_i}\myvec e_j=\sum_k \Gamma^k_{ij}\myvec e_k$.
By the Leibnitz rule for differentiation this implies that
\begin{equation}
  \label{eq:dvu}
  \grad_{\myvec v}\myvec u=\sum_{i,j}  v^i\tppp{u^j}{p_i}\myvec e_j+
\sum_{i,j,k} v^iu^j\Gamma^k_{ij}\myvec e_k.
\end{equation}
The rule for differentiating vectors extends to a rule for
differentiating arbitrary tensors by the requirements of linearity,
compatibility with tensor contraction, and the Leibnitz product rule.
On a Riemannian manifold the Christoffel symbols are uniquely determined by the requirement 
that the metric be invariant, i.e., that its derivative vanishes. This requirement yields (e.g. \cite[\S86]{MR143451}, \cite[pp. 96\textendash97]{MR4875093})
\begin{equation}
  \label{eq:gamg}
  \Gamma^c_{ab}=\tfrac12\sum_d g^{dc}(\prt_a g_{bd}+\prt_b g_{ad}-\prt_d g_{ab}),
\end{equation}
where the matrix $g^{dc}$ is the inverse of the matrix $g_{dc}$ of the metric tensor.
For the sphere both $g$ and its inverse
\begin{equation}
  \label{eq:ginv}
   \begin{pmatrix}
    g^{\phi,\phi}&g^{\phi,\theta}\\ g^{\theta,\phi}&g^{\theta,\theta}
  \end{pmatrix}
=
\begin{pmatrix}
  \tfrac1{\sin^2\!\theta}&0\\0&1
\end{pmatrix},
\end{equation}
are diagonal, and $\prt_\alpha g_{\beta\gamma}$ is nonzero only when
$\alpha=\theta$ and $\beta=\gamma=\phi$, in which case $\prt_\theta g_{\phi\phi}=2\sin\?\theta\cos\?\theta$.
Together with \eqref{eq:gamg}\textendash\eqref{eq:ginv}, this implies that
the only nonzero components of the Christoffel tensor for the sphere are
\begin{equation*}
  \Gamma^\phi_{\phi\theta}=\Gamma^\phi_{\theta\phi}=\cot\?\theta,
\qquad\text{and}\qquad \Gamma^\theta_{\phi\phi}=-\sin\?\theta\cos\?\theta.
\end{equation*}
By \eqref{eq:dvu}, the formula for covariant differentiation of a vector field 
on the sphere is therefore
\begin{equation}
  \label{eq:covderiv}
\begin{aligned}
  \grad_{\myvec v}\myvec u&=
  \lp
  v_\theta\tppp{u_\theta}{\theta}+v_\phi\tppp{u_\theta}{\phi}-
  v_\phi u_\phi\sin\?\theta\cos\?\theta\rp\!
  \myvec e_\theta
\\&\qquad
+ 
\lp v_\theta\tppp{u_\phi}{\theta}+v_\phi\tppp{u_\phi}{\phi}
+\lb v_\theta u_\phi+ v_\phi u_\theta\rb \cot\?\theta\rp\!
  \myvec e_\phi.
\end{aligned}
\end{equation}

\subsection{Differential forms and musical isomorphisms}

One-forms are the duals to vector fields. On the sphere a one-form is an expression $\alpha\eqdef a d\phi+b d\theta$, whose value on a vector field
$\myvec v\eqdef c \myvec e_\phi+d \myvec e_\theta$ is
\begin{equation}\label{eq:dual}
  \alpha[\myvec v]=(a d\phi+b d\theta)[c \myvec e_\phi+d \myvec e_\theta]=a c+b d.
\end{equation}
Two-forms on the sphere are expressions $\beta\eqdef f d\phi\wedge d\theta$. They can be constructed from one-forms in two ways. First, they can be
constructed by using the wedge product $\wedge$, which satisfies the anti-symmetry properties
\begin{equation}\label{eq:wedgeanti}
  d\theta\wedge d\phi=-d\phi\wedge d\theta,\quad d\phi\wedge d\phi=0=d\theta\wedge d\theta.
\end{equation}
The formulas in~\eqref{eq:wedgeanti} together with linearity determine the wedge product of any two one-forms. Second, a two form can be obtained from
a one-form via differentiation, according to the rule
\begin{equation}\label{eq:d1}
  d[f d\alpha]=\tppp{f}{\phi}d\phi\wedge \alpha+\tppp{f}{\theta} d\theta\wedge \alpha.
\end{equation}
Similarly, a function will also be called a zero-form, and applying $d$ to a zero-form yields a one-form via
\begin{equation}\label{eq:d0}
  df=\tppp{f}{\phi} d\phi +\tppp{f}{\theta}d\theta.
\end{equation}
On the other hand, since the sphere is two-dimensional there are no three-forms, so applying $d$ to any two-form yields zero, i.e.,
\begin{equation}
  \label{eq:d2}
  d[f d\phi\wedge d\theta]=0.
\end{equation}
The number zero, one, or two preceding the word form is called the degree of the form.

The ``musical'' operators $\flat$ (flat) and $\sharp$ (sharp) convert vector fields to one-forms and vice-versa, respectively, via the rules
\cite[pp. 27\textendash28]{MR1468735} 
\begin{equation}\label{eq:music}
(\flat[\myvec u])[\myvec v]\eqdef g(\myvec u,\myvec v)\qquad\text{and}\qquad \sharp\flat[\myvec u]=\myvec u.
\end{equation}
In view of
\eqref{eq:sphereip} and \eqref{eq:dual}, for the sphere this yields \cite[A.4]{MR2998607}
\begin{align}
\label{eq:flat}
  \flat[f \myvec e_\phi]&=f \sin^2\!\theta d\phi,& \flat[f \myvec e_\theta]&=f d\theta
\\
\label{eq:sharp}
  \sharp[f d\phi]&= \tfrac{f}{\sin^2\!\theta}\myvec e_\phi, & \sharp[f d\theta]&=f \myvec e_\theta,
\end{align}
which are extended to all vector fields and one-forms, respectively, by linearity, and imply that 
\begin{equation}\label{eq:flatsharp}
\text{$\flat\sharp$ is also the identity operator.}
\end{equation}

\subsection{Star operator and Hodge Laplacian}

The star operator $\star$ maps functions into two-forms and vice-versa, and maps one-forms into other one-forms, via linearity plus the
rules~\cite[A.4]{MR2998607}
\begin{align}
\label{eq:star20}
  \star[f d\phi\wedge d\theta]&=\tfrac{f}{\sin\?\theta}, & \star[f]&=f\sin\?\theta d\phi\wedge d\theta
\\\
  \star[f d\phi]&=\tfrac{f}{\sin\?\theta}d\theta, & \star[f d\theta]&=-f\sin\?\theta d\phi.
\label{eq:star1}\end{align}
The definitions \eqref{eq:star1} imply that
\begin{equation}
  \label{eq:starstar}
  \star\,\star=-1 \qquad \text{when applied to one-forms.}
\end{equation}

For two-dimensional manifolds like the sphere, the co-differentiation operator $\delta$ is defined on forms via 
\cite[p. 200 Equation (2)]{MR722297}
\begin{equation}\label{eq:delta-def}
  \delta[\alpha]:=-*d*.
\end{equation}
This implies that co-differentiation maps forms of degree $k$ to forms of degree $k-1$, and in particular
\begin{equation}\label{eq:delta-f-0}
 \delta[f]=0.
\end{equation}
Calculations using \eqref{eq:delta-def} and the definitions of $\star$ and $d$ yield
\begin{equation}\label{eq:delta}
\begin{aligned}
  \delta[f d\phi\wedge d\theta]&= \tppp{f}{\theta}d\phi-f \cot\?\theta d\phi-\csc^2\!\theta \tppp{f}{\phi}d\theta   ,
\\
  \delta[f d\phi+h d\theta]&=-h \cot\?\theta-\tppp{h}{\theta} -\csc^2\!\theta \,\tppp{f}{\phi}.
\end{aligned}
\end{equation}
Formulas \eqref{eq:d1}, \eqref{eq:d0}, and \eqref{eq:delta} imply that the composition of either $d$ or $\delta$ with itself vanishes identically:
\begin{equation}\label{eq:dd0}
 d\, d=0= \delta\,\delta.
\end{equation}

The Hodge Laplacian is the operator \cite[p. 200 Equation (3)]{MR722297}
\begin{equation}
  \label{eq:hodge}
  \Delta_H\eqdef d \delta+\delta d=-(\star\, d\star d+d \star d\, \star),
\end{equation}
which takes forms of any degree into forms of the same degree. The second formula in \eqref{eq:hodge} follows from
\eqref{eq:delta-def}. Calculations using \eqref{eq:hodge}, \eqref{eq:d1}\textendash\eqref{eq:d2}, and
\eqref{eq:star20}\textendash\eqref{eq:star1} yield
\begin{align}
\label{eq:hodge0}
  \Delta_H f&=-\tpppp{f}{\theta}-\csc^2\!\theta\, \tpppp{f}{\phi}-\cot \?\theta \tppp{f}{\theta},
\\[3pt]
\label{eq:hodge1}
  \Delta_H[f d\phi+h d\theta]&=\lp -\tpppp{f}{\theta}-\csc^2\!\theta\,\tpppp{f}{\phi}+\cot\?\theta\, \tppp{f}{\theta}-2\cot\?\theta\,\tppp{h}{\phi}\rp d\phi
\\&\quad+\lp-\tpppp{h}{\theta}-\csc^2\!\theta\,\tpppp{h}{\phi}-\cot\?\theta\,\tppp{h}{\theta}+2\cot\?\theta\csc^2\!\theta\,\tppp{f}{\phi}
+h\csc^2\!\theta\rp d\theta
\end{align}

\subsection{Differential operators on vector fields}

By using the musical operators it is possible to translate the above differential operators on forms into differential operators on vector fields. 
Define the following operators:
\begin{align}
\label{eq:divcurl}
  \div \myvec u&\eqdef \star\, d \star \flat[\myvec u],
&
\curl \myvec u&\eqdef \star\, d\,\flat[\myvec u],
\\
\label{eq:gradperp}
\grad f&\eqdef \sharp\,d [f],
&
\myvec u^\perp&\eqdef -\sharp\star\flat[\myvec u], 
&
\grad^\perp f&\eqdef (\grad f)^\perp =-\sharp\star d[f],
\\
\label{eq:laplacian}
\Delta \myvec u&\eqdef \grad\div \myvec u-\grad^\perp\curl \myvec u,
&
\Delta f&\eqdef \div\grad f.
\end{align}
In view of \eqref{eq:flatsharp}, \eqref{eq:starstar}, \eqref{eq:dd0}, and \eqref{eq:music}, the definitions \eqref{eq:divcurl} and \eqref{eq:gradperp}
imply that
\begin{equation}
  \label{eq:divgradperp0}
  \div\grad^\perp =0 \qquad\text{and}\qquad (\myvec u^\perp)^\perp =-\myvec u.
\end{equation}
The definition of the Laplacian on functions agrees with the definition of the Hodge Laplacian $\Delta_H$ except for its sign, i.e.,
\begin{equation}\label{eq:lap=hodge:f}
  \Delta f=-\Delta_H f,
\end{equation}
and the Laplacian on vector fields is similarly compatible with the Hodge Laplacian and musical isomorphisms, i.e.
\begin{equation}
  \label{eq:lap=hodge:u}
  \Delta \myvec u= -\sharp\Delta_H\flat [\myvec u],
\end{equation}
as can be seen by using \eqref{eq:hodge}, \eqref{eq:delta-f-0}, and \eqref{eq:divcurl}\textendash\eqref{eq:laplacian}.
Combining \eqref{eq:divcurl} with \eqref{eq:flat}, \eqref{eq:star1}, and \eqref{eq:d1} yields the formula
\begin{align}
  \label{eq:div}
\div (u_\phi \myvec e_\phi+u_\theta \myvec e_\theta)&
=\tppp{u_\phi}{\phi}+\tfrac1{\sin\?\theta}\prt_\theta (\sin\?\theta\, u_\theta)
\end{align}
for the divergence of a vector field on the sphere, 
combining \eqref{eq:gradperp} with \eqref{eq:d0} and \eqref{eq:sharp} yields
the formula
\begin{equation}
  \label{eq:grad}
  \grad f=\tfr{\sin^2\!\theta}\tppp{f}{\phi}\myvec e_\phi+\tppp{f}{\theta} \myvec e_\theta
\end{equation}
for the gradient of a function on the sphere, and 
combining \eqref{eq:lap=hodge:u}, \eqref{eq:hodge1}, and \eqref{eq:flat}\textendash\eqref{eq:sharp} shows that the Laplacian of a vector field is
\begin{equation}
  \label{eq:lapu}
\begin{aligned}
  \Delta(u \myvec e_\phi+v \myvec e_\theta)&=
\lb \tpppp{u}{\theta}+\csc^2\!\theta\, \tpppp{u}{\phi}+3\cot\?\theta\, \tppp{u}{\theta}+2\cot\?\theta\csc^2\!\theta\, \tppp{v}{\phi} -2u\rb \myvec e_\phi
\\&\quad+\lb \tpppp{v}{\theta}+\csc^2\!\theta\,\tpppp{v}{\phi}+\cot\?\theta\,\tppp{v}{\theta}
-2\cot\?\theta\,\tppp{u}{\phi}-\csc^2\!\theta\, v\rb \myvec e_\theta.
\end{aligned}
\end{equation}
Also, combining \eqref{eq:flat}, \eqref{eq:star1}, and \eqref{eq:sharp} shows that the perpendicular of a
vector on the sphere is given by
\begin{equation}
  \label{eq:uperp}
  (u_{\phi}\myvec e_\phi+u_\theta \myvec e_\theta)^\perp=\tfr{\sin\?\theta} \,u_\theta \myvec e_\phi-\sin\?\theta\,u_\phi \myvec e_\theta,
\end{equation}
which by \eqref{eq:g} implies that
\begin{equation}
  \label{eq:uperpperpu}
  g(\myvec u,\myvec u^\perp)=0.
\end{equation}
It will sometimes be useful to consider $\perp$ as an operator that can be applied to vectors, which in view of \eqref{eq:gradperp} and
\eqref{eq:uperp} has the form
\begin{equation}
  \label{eq:perp}
  \perp\eqdef -\sharp\star\flat =\lp\begin{smallmatrix} \phantom{-}0 &\tfr{\sin\?\theta}\\-\sin\?\theta&0\end{smallmatrix}\rp.
\end{equation}

\subsection{Integration and adjoints}

According to the standard theory of integration on manifolds, the integral of a function on the sphere of a measurable function $f$ is 
\begin{equation}
  \label{eq:int-on-S}
\int_{\mathbb S^2} f\eqdef   \int_{\mathbb S^2} f(\phi,\theta) \sqrt{\det g}\, d\phi\wedge d\theta
=\int_{\theta=0}^\pi\int_{\phi=0}^{2\pi} f(\phi,\theta)\sin\?\theta \,d\phi\,d\theta,
\end{equation}
where $g$ with no arguments denotes the matrix of metric tensor components appearing in \eqref{eq:sphereip}.
Using once more the metric tensor $g(\cdot,\cdot)$ defined in~\eqref{eq:veclen}\textendash\eqref{eq:sphereip},
we can therefore define an $L^2$ inner product on pairs~$\lp\begin{smallmatrix}\myvec u\\h\end{smallmatrix}\rp$ consisting of a vector and a scalar
by
\begin{equation}
  \label{eq:ip}
\begin{aligned}
  &\ip{\lp\begin{smallmatrix}\myvec u\\h\end{smallmatrix}\rp}{\lp\begin{smallmatrix}\myvec v\\k\end{smallmatrix}\rp}
\eqdef \int_{\mathbb S^2} [g(\myvec u,\myvec v)+hk]\sqrt{\det g}\,d\phi\wedge d\theta
\\&=
 \int_{\theta=0}^\pi\int_{\phi=0}^{2\pi}  \lb \sin^2\!\theta \,u_\phi(\phi,\theta)v_\phi(\phi,\theta)+
  u_\theta(\phi,\theta)v_\theta(\phi,\theta)+h(\phi,\theta)k(\phi,\theta)\rb \sin\?\theta \,d\phi\,d\theta.
\end{aligned}
\end{equation}
An operator $L$ then has a formal adjoint operator $L^*$ if
\begin{equation}
  \label{eq:adjcond}
  \ip{L\lp\begin{smallmatrix}\myvec u\\h\end{smallmatrix}\rp}{\lp\begin{smallmatrix}\myvec v\\k\end{smallmatrix}\rp}
\equiv \ip{\lp\begin{smallmatrix}\myvec u\\h\end{smallmatrix}\rp}{L^*\lp\begin{smallmatrix}\myvec v\\k\end{smallmatrix}\rp}.
\end{equation}
As usual, an operator is symmetric if $L^*=L$ when acting on smooth functions, and is antisymmetric if $L^*=-L$ for such functions.

Calculations using  \eqref{eq:g}, \eqref{eq:div}, \eqref{eq:int-on-S}, and \eqref{eq:ip}
 show that for any smooth scalar field~$k$ and vector field~$\myvec v$,
$\int_{\mathbb S^2}\div \myvec v=0$ and
\begin{equation}
  \label{eq:divstargrad}
  \ip{\lp\begin{smallmatrix}\myvec v\\k\end{smallmatrix}\rp}{\lp\begin{smallmatrix}\grad h\\\div \myvec v\end{smallmatrix}\rp}=
\int_{\mathbb S^2}\div(k\myvec v)=0.
\end{equation}

\bibliographystyle{plain}
\bibliography{sphere}
\end{document}